\documentclass[12pt, reqno]{amsart}

\usepackage{amsmath}
\usepackage{amsthm}
\usepackage{amssymb}
\usepackage[abbrev]{amsrefs}
\usepackage{mathrsfs}
\usepackage[usenames]{color}
\usepackage{ifthen}
\usepackage{yhmath}
\usepackage{graphicx}
\usepackage[all]{xy}
\usepackage{bm}
\usepackage{bbm}
\usepackage{enumitem}
\usepackage{cases}

\makeatletter
\newcommand{\leqnos}{\tagsleft@true\let\veqno\@@leqno}
\newcommand{\reqnos}{\tagsleft@false\let\veqno\@@eqno}
\reqnos
\makeatother

\AtBeginDocument{
\def\MR#1{}
}

\newtheorem{thm}{}[section]
\newtheorem{theorem}[thm]{Theorem}

\newtheorem{lemma}[thm]{Lemma}
\newtheorem{proposition}[thm]{Proposition}

\theoremstyle{definition}
\newtheorem{definition}[thm]{Definition}
\theoremstyle{remark}

\newtheorem{remark}[thm]{Remark}
\newtheorem{question}[thm]{Question}

\numberwithin{equation}{section}
\allowdisplaybreaks

\newcommand{\VMO}{\ensuremath{\mathrm{VMO}}}
\newcommand{\BMO}{\ensuremath{\mathrm{BMO}}}
\newcommand{\BV}{\ensuremath{\mathrm{BV}}}
\newcommand{\Ts}{\ensuremath{\mathcal{T}}}
\newcommand{\tl}{\ensuremath{\bm{f}}}

\newcommand{\Qy}{\ensuremath{\mathscr{R}}}
\newcommand{\Dy}{\ensuremath{\mathscr{D}}}
\newcommand{\AD}{\ensuremath{\mathcal{A}}}
\newcommand{\HB}{\ensuremath{\mathcal{H}}}
\newcommand{\SB}{\ensuremath{\mathcal{S}}}
\newcommand{\DB}{\ensuremath{\mathcal{D}}}
\newcommand{\sss}{\ensuremath{\bm{s}}}
\newcommand{\dd}{\ensuremath{\bm{d}}}
\newcommand{\ee}{\ensuremath{\bm{e}}}
\newcommand{\JB}{\ensuremath{\mathcal{J}}}
\newcommand{\FieldC}{\ensuremath{{\bm{\Upsilon}}}}
\newcommand{\prim}{\ensuremath{\bm{\sigma}}}
\newcommand{\dom}{\ensuremath{\bm{\delta}}}
\newcommand{\ldf}{\ensuremath{\underline{\bm{\varphi}}}}
\newcommand{\udf}{\ensuremath{\overline{\bm{\varphi}}}}
\newcommand{\lsdf}{\ensuremath{\bm{\varphi_l}}}
\newcommand{\usdf}{\ensuremath{\bm{\varphi_u}}}
\newcommand{\Ct}{\ensuremath{\bm{C}}}

\newcommand{\wbid}{\ensuremath{\bm{B^w}}}
\newcommand{\bid}{\ensuremath{\bm{B}}}
\newcommand{\rto}{\ensuremath{\bm{r}}}
\newcommand{\ucc}{\ensuremath{\bm{u}}}
\newcommand{\qglc}{\ensuremath{\bm{q}}}
\newcommand{\qg}{\ensuremath{\bm{g}}}
\newcommand{\qgc}{\ensuremath{\bm{g^c}}}
\newcommand{\fss}{\ensuremath{\bm{\psi}}}
\newcommand{\dr}{\ensuremath{\bm{a}}}
\newcommand{\sqd}{\ensuremath{\bm{\lambda^d}}}
\newcommand{\sq}{\ensuremath{\bm{\lambda}}}
\newcommand{\unc}{\ensuremath{\bm{k}}}
\newcommand{\uncc}{\ensuremath{\bm{k^c}}}
\newcommand{\leba}{\ensuremath{\bm{L^a}}}
\newcommand{\leb}{\ensuremath{\bm{L}}}
\newcommand{\sdem}{\ensuremath{\bm{\mu^s}}}

\newcommand{\dem}{\ensuremath{\bm{\mu}}}
\newcommand{\slcd}{\ensuremath{\bm{\nu^d}}}
\newcommand{\slc}{\ensuremath{\bm{\nu}}}
\newcommand{\xx}{\ensuremath{\bm{x}}}
\newcommand{\yy}{\ensuremath{\bm{y}}}
\newcommand{\XX}{\ensuremath{\mathbb{X}}}
\newcommand{\YY}{\ensuremath{\mathbb{Y}}}
\newcommand{\XB}{\ensuremath{\mathcal{X}}}
\newcommand{\Ind}{\ensuremath{\mathbbm{1}}}
\newcommand{\BB}{\ensuremath{\mathcal{B}}}
\newcommand{\ww}{\ensuremath{\bm{w}}}
\newcommand{\Id}{\ensuremath{\mathrm{Id}}}
\newcommand{\GG}{\ensuremath{\mathcal{G}}}

\newcommand{\TT}{\ensuremath{\mathbb{T}}}
\newcommand{\EE}{\ensuremath{\mathbb{E}}}
\newcommand{\Tt}{\ensuremath{\mathcal{T}}}
\newcommand{\FF}{\ensuremath{\mathbb{F}}}
\newcommand{\RR}{\ensuremath{\mathbb{R}}}
\newcommand{\CC}{\ensuremath{\mathbb{C}}}
\newcommand{\NN}{\ensuremath{\mathbb{N}}}
\newcommand{\ZZ}{\ensuremath{\mathbb{Z}}}
\newcommand{\RTO}{\ensuremath{\mathcal{R}}}
\newcommand{\Fou}{\ensuremath{\mathcal{F}}}
\newcommand{\YB}{\ensuremath{\mathcal{Y}}}

\DeclareMathOperator*{\Ave}{Ave}
\DeclareMathOperator{\spn}{span}
\DeclareMathOperator{\sgn}{sign}
\DeclareMathOperator{\supp}{supp}

\title[Greedy-like parameters of second generation]{New parameters and Lebesgue-type estimates in greedy approximation}

\begin{document}

\author[Albiac]{Fernando Albiac}
\address{Department of Mathematics, Statistics, and Computer Sciencies--InaMat2 \\
Universidad P\'ublica de Navarra\\
Campus de Arrosad\'{i}a\\
Pamplona\\
31006 Spain}
\email{fernando.albiac@unavarra.es}

\author[Ansorena]{Jos\'e L. Ansorena}
\address{Department of Mathematics and Computer Sciences\\
Universidad de La Rioja\\
Logro\~no\\
26004 Spain}
\email{joseluis.ansorena@unirioja.es}

\author[Bern\'a]{Pablo M. Bern\'a}
\address{Pablo M. Bern\'a\\
Departamento de Matem\'atica Aplicada y Es\-ta\-d\'is\-ti\-ca, Facultad de Ciencias Econ\'omicas y Empresariales, Universidad San Pablo-CEU, CEU Universities\\ Madrid, 28003 Spain.}
\email{pablo.bernalarrosa@ceu.es}



\begin{abstract}
The purpose of this paper is to quantify the size of the Lebesgue constants $(\leb_{m})_{m=1}^{\infty}$ associated with the thresholding greedy algorithm in terms of a new generation of parameters that modulate accurately some features of a general basis. This fine-tuning of constants allows us to provide an answer to the question raised by Temlyakov in 2011 to find a natural sequence of greedy-type parameters for arbitrary bases in Banach (or quasi-Banach) spaces which combined \emph{linearly} with the sequence of unconditionality parameters $(\unc_m)_{m=1}^{\infty}$ determines the growth of $(\leb_{m})_{m=1}^{\infty}$. Multiple theoretical applications and computational examples complement our study.
\end{abstract}

\thanks{F. Albiac acknowledges the support of the Spanish Ministry for Science and Innovation under Grant PID2019-107701GB-I00 for \emph{Operators, lattices, and structure of Banach spaces}. F. Albiac and J.~L. Ansorena acknowledge the support of the Spanish Ministry for Science, Innovation, and Universities under Grant PGC2018-095366-B-I00 for \emph{An\'alisis Vectorial, Multilineal y Aproximaci\'on}.}

\maketitle

\section{Introduction and background}\noindent
Let $\XX$ be an infinite-dimensional separable Banach space (or, more generally, a quasi-Banach space) over the real or complex field $\FF$, and let $\XB=(\xx_n)_{n=1}^\infty$ be a \emph{basis} in $\XX$, i.e., $(\xx_n)_{n=1}^\infty$ is a norm-bounded sequence that generates the entire space $\XX$, in the sense that
\[
\overline{\spn(\xx_n \colon n\in\NN)}=\XX,
\]
and for which there is a (unique) norm-bounded sequence $\XB^*=(\xx_{n}^*)_{n=1}^\infty$ in the dual space $\XX^{\ast}$ such that $(\xx_{n}, \xx_{n}^{\ast})_{n=1}^{\infty}$ is a biorthogonal system. The sequence $\XB^*$ will be called the \emph{dual basis} of $\XB$.

For each $m\in\NN$, we let $\Sigma_{m}[\XB, X]$ denote the collection of all $f$ in $\XX$ which can be expressed as a linear combination of $m$ elements of $\XB$, that is,
\[
\Sigma_{m}[\XB, \XX]= \left\{\sum_{n\in A} a_n\, \xx_n \colon A\subseteq
\NN,\, |A| = m,\; a_n\in\RR\right\}, \quad m=1,2,\dots
\]
A fundamental question in nonlinear approximation theory using bases is how to construct for each $f\in \XX$ and each $m\in\NN$ an element $g_{m}$ in $\Sigma_{m}$ so that the error of the approximation of $f$ by $g_{m}$ is as small as possible. To that end we need, on one hand, an easy way to build for all $m\in \NN$ an $m$-term approximant of any function (or signal) $f\in \XX$, and on the other hand, a way to measure the efficiency of our approximation.

Konyagin and Temlyakov \cite{KoTe1999} introduced in 1999 the \emph{thresholding greedy algorithm} (TGA for short) $(\GG_{m})_{m=1}^{\infty}$, where $\GG_{m}(f)$ is obtained by choosing the first $m$ terms in decreasing order or magnitude from the formal series expansion $\sum_{n=1}^\infty\xx_n^*(f)\, \xx_n$ of $f$ with respect to $\XB$, with the agreement that when two terms are of equal size we take them in the basis order.
By our assumptions on the dual basis $\XB^*$, the \emph{coefficient transform}
\[
\Fou\colon\XX\to\FF^\NN, \quad
\Fou(f)=(\xx_n^*(f))_{n=1}^\infty, \quad f\in\XX,
\]
is a bounded map from $\XX$ into $c_0$, so that the maps $\GG_{m}\colon \XX\to \XX$ are well defined for all $m\in \NN$; however, the operators $(\GG_{m})_{m=1}^{\infty}$ are not linear nor continuous.

To measure the performance of the greedy algorithm we compare the error $\Vert f-\GG_{m}(f)\Vert$ in the approximation of any $f\in \XX$ by $\GG_{m}(f)$, with the \emph{best $m$-term approximation error}, given by
\[
\sigma_{m}[\XB, \XX](f):=\sigma_{m}(f) =\inf\{ \Vert f -g \Vert \colon g\in \Sigma_{m}\}.
\]
An upper estimate for the rate $\Vert f- \GG_m(f)\Vert/\sigma_{m}(f)$ is usually called a \emph{Lebesgue-type inequality} for the TGA (see \cite{Temlyakov2015}*{Chapter 2}). Obtaining Lebesgue-type inequalities is tantamount to finding upper bounds for the \emph{Lebesgue constants} of the basis, given for $m\in \NN$ by
\[
\leb_m=\leb_m[\XB,\XX]=\sup\left\{\frac{\Vert f- \GG_{m}(f)\Vert}{\sigma_{m}(f)}\colon f\in \XX\setminus \Sigma_m\right\}.
\]
By definition, the basis $\XB$ is \emph{greedy} \cite{KoTe1999} if and only if
\[
\Ct_g=\Ct_g[\XB,\XX]:=\sup_m \leb_m<\infty.
\]
Certain important bases, such as the Haar system in $L_p$, $1<p<\infty$, or the Haar system in $H_p$, $0<p\le 1$, are known to be greedy (see \cites{Temlyakov1998,Woj2000,AABW2019}). In the literature we also find instances where the Lebesgue constants have been computed or estimated for special non-greedy bases in important spaces:

\begin{itemize}[label=\tiny{$\bullet$}, leftmargin=*]
\item In \cite{Oswald2001}, the Lebesgue parameters of the Haar basis
in the spaces $\BMO$ and dyadic $\BMO$
were computed.
\item More recently, in \cites{TemYangYe1, TemYangYe2}, the Lebesgue constants for tensor product bases in $L_{p}$-spaces (in particular, for the multi-Haar basis) were calculated.
\item The Lebesgue constants for the trigonometric basis in $L_{p}$ are also known (see \cite{Temlyakov1998b}).

\item The paper \cite{DSBT2012} estimates the Lebesgue constants for bases in $L_{p}$-spaces with specific properties (such as being uniformly bounded).
\item Lebesgue constants for redundant dictionaries are studied in \cite{Tembook}*{Section 2.6}.
\end{itemize}

Calculating the exact value of the Lebesgue constants can be in general a difficult task, so to study the efficiency of non-greedy bases we must settle for obtaining easy-to-handle estimates that control the asymptotic growth of $(\leb_m)_{m=1}^{\infty}$. To center the problem we shall introduce some preliminary notation.

\subsection{Unconditionality parameters and democracy parameters}
Konyagin and Temlyakov \cite{KoTe1999} characterized greedy bases as those bases that are simultaneously unconditional and democratic. Thus, in order to find bounds for the Lebesgue constants it is only natural to quantify the unconditionality and the democracy of a basis and study their relation with $\leb_m$.

For finite $A\subseteq \NN$, we let $S_A=S_A[\XB,\XX]\colon \XX \to\XX$ denote the coordinate projection on the set $A$ , i.e.,
\[
S_A(f)=\sum_{n\in A} \xx_n^*(f)\, \xx_n,\quad f\in \XX.
\]

The norm of the coordinate projections in a basis $\XB=(\xx_n)_{n=1}^\infty$ is quantified by the \emph{unconditionality parameters}
\[
\unc_m=\unc_m[\XB,\XX] :=\sup_{|A|= m} \Vert S_A\Vert, \quad m\in\NN,
\]
and by the \emph{complemented unconditionality parameters},
\[
\uncc_m=\unc_m[\XB,\XX] :=\sup_{|A|= m} \Vert \Id_\XX-S_A\Vert, \quad m\in\NN.
\]
Note that, if $\XX$ is a $p$-Banach space, $0<p\le 1$, then 
\begin{equation}\label{kmkmc}
(\unc_m)^p\le 1+(\uncc_m)^p,\quad (\uncc_m)^p\le 1+(\unc_m)^p,
\end{equation} and $(\unc_{2m})^p \le 2 (\unc_m)^p$ for all $m\in\NN$.

We also define the $m$th \emph{democracy parameter} of the basis as
\[
\dem_m=\dem_m[\XB,\XX]=\sup\limits_{|A|=|B|\le m}\frac{\Vert\Ind_A\Vert}{\Vert \Ind_B\Vert},
\]
where
\[
\Ind_A=\Ind_A[\XB,\XX]=\sum_{n\in A} \xx_n.
\]
A basis is \emph{unconditional} if and only if $\sup_m\unc_m<\infty$, and it is \emph{democratic} if and only if $\sup_m \dem_m<\infty$. A reproduction of the original proof of the above-mentioned characterization of greedy bases from \cite{KoTe1999} for general bases, where we pay close attention to thfe dependency on $m$ in the constants involved, yields the following upper and lower bounds for the Lebesgue parameters in terms of $\unc_m$ and $\dem_m$:
\begin{equation}\label{eq:UncG}
\frac{1}{C_2} \max\left\{ \unc_m, \dem_m\right\} \le \leb_m \le C_1 \unc_m\, \dem_m,
\end{equation}
where $C_1$ and $C_2$ depend only on the modulus of concavity of the space $\XX$ (see \cite{BBG2017}*{Proposition 1.1} and \cite{AABW2019}*{Theorem 7.2}). However, these bounds are not optimal for two reasons: first of all, since the function on the left is not the same as the function on the right, we loose accuracy in estimating the size of $\leb_{m}$; secondly, the function on the right does not depend linearly on the unconditionality and democracy parameters.

The investigation of Lebesgue constants for greedy algorithms dates back to the initial stages of the theory, with some relevant ideas appearing already in \cite{KoTe1999}. Oswald gave in \cite{Oswald2001}*{Theorem 1} the correct asymptotic behavior for the quantities $\leb_m$ in the general case
replacing $(\dem_m)_{m=1}^\infty$ with other parameters. However, its application to particular systems is tedious due the complicated, implicit definitions of the parameters his estimates rely on.

Other authors have approached the subject by imposing extra conditions on the basis which permit to obtain sharp estimates for $(\leb_m)_{m=1}^{\infty}$. The first movers in this direction were Garrig\'os et al., who in 2013 gave the following partial answer to Temlyakov's question for \emph{quasi-greedy} bases, i.e., bases for which the operators $(\GG_m)_{m=1}^\infty$ are uniformly bounded (or, equivalently, bases for which $\GG_{m}(f)$ converges to $f$ for all $f\in \XX$).

\begin{theorem}[\cite{GHO2013}*{Theorem 1.1}]\label{GHOthm} If the basis $\XB$ is quasi-greedy then there is a constant $C$ such that
\[
\frac{1}{C} \max\{\dem_m, \unc_m\}\le \leb_{m}\le C \max\{\dem_m, \unc_m\}, \quad m\in \NN.
\]
\end{theorem}
Thus, in the particular case that $\XB$ is unconditional, $(\leb_m)_{m=1}^\infty$ is asymptotically of the same order as $(\dem_m)_{m=1}^\infty$.

Another important property of quasi-greedy bases is that their unconditionality constants grow slowly. Indeed, if $\XX$ is a $p$-Banach space, the contribution of $\unc_m$ to $\leb_{m}$ is at most of the order of $(\log m)^{1/p}$
(see \cite{DKK2003}*{Lemma 8.2}, \cite{GHO2013}*{Theorem 5.1} and \cite{AAW2020}*{Theorem 5.1}).
Hence, if $\XB$ is quasi-greedy and democratic there is a constant $C$ such that
\[
\leb_m \le C (\log m)^{1/p},\quad m\ge 2.
\]

\subsection{Towards new parameters in greedy approximation}
This section is geared towards the introduction of a new breed of parameters with the aim to provide a satisfactory answer to Temlyakov's aforementioned question. For that we will adopt a view point that regards certain Lebesgue-type parameters as quantifiers of the different degrees of symmetry that can be found in a basis.

Recall that a basis $\XB=(\xx_n)_{n=1}^\infty$ of $\XX$ is \emph{symmetric} if it is equivalent to all its permutations, i.e., there is a constant $C\ge 1$ such that
\begin{equation}\label{eq:sym}
\frac{1}{C} \left \Vert \sum_{n=1}^\infty a_n \, \xx_n\right\Vert
\le \left\Vert \sum_{n=1}^\infty a_n \, \xx_{\pi(n)}\right\Vert
\le C \left \Vert \sum_{n=1}^\infty a_n \, \xx_n\right\Vert
\end{equation}
for all $(a_n)_{n=1}^\infty\in c_{00}$ and all permutations $\pi$ on $\NN$. Democracy is the weakest symmetry condition that a basis can have, where we demand to a basis to verify \eqref{eq:sym} when all coefficients $a_{n}$ of $f\in\XX$ are equal (without loss of generality) to 1.

Symmetric bases are in particular unconditional, which permits to improve the previous inequality to have
\begin{equation}\label{eq:symbis}
\frac{1}{C} \left \Vert \sum_{n=1}^\infty a_n \, \xx_n\right\Vert
\le \left\Vert \sum_{n=1}^\infty \varepsilon_n\, a_{n} \, \xx_{\pi(n)}\right\Vert
\le C \left \Vert \sum_{n=1}^\infty a_n \, \xx_n\right\Vert
\end{equation}
for all $(\varepsilon_n)_{n=1}^\infty\in \EE^\NN$, where $\EE$ denotes the subset of $\FF$ consisting of all scalars of modulus 1, for a possibly larger constant $C$. If a basis $\XB=(\xx_n)_{n=1}^\infty$ of $\XX$ satisfies \eqref{eq:symbis} when all coefficients $a_{n}$ are 1 it is called \emph{superdemocratic}. To quantify the superdemocracy of $\XB$ we use the $m$th \emph{super-democracy parameter},
\[
\sdem_m=\sdem_m[\XB,\XX]=\sup\left\{\frac{\Vert \Ind_{\varepsilon,A} \Vert}{\Vert \Ind_{\delta,B} \Vert}\colon |A|=|B|\le m, \varepsilon\in\EE^A,\; \delta\in\EE^B\right\},
\]
where
\[
\Ind_{\varepsilon,A}=\Ind_A[\XB,\XX]=\sum_{n\in A} \varepsilon_n \, \xx_n,
\]
so that $\XB$ is \emph{super-democratic} if $\sup_m \sdem_m<\infty$. 

Although the parameters $(\sdem_m)_{m=1}^{\infty}$ are one step up in the scale of symmetry, in practice they do not provide asymptotically better estimates than $(\dem_m)_{m=1}^{\infty}$ for $(\leb_m)_{m=1}^{\infty}$. Indeed, with a smaller constant $C_1$ and a larger constant $C_2$, we have
\[
\frac{1}{C_2} \max\left\{ \unc_m, \sdem_m\right\} \le \leb_m \le C_1 \unc_m \, \sdem_m, \quad m\in\NN
\]
(cf. \cite{BBG2017}*{Proposition 1.1}).

Another condition related to symmetry which has been successfully implemented in the theory is the so-called symmetry for largest coefficients. Let
\[
\supp(f)=\{ n\in\NN \colon \xx_n^*(f) \not=0\}
\]
denote the support of $f\in\XX$ with respect to the basis $\XB$. We define
\[
\slc_m=\slc_m[\XB,\XX]= \sup\frac{\Vert \Ind_{\varepsilon,A} +f \Vert}{ \Vert \Ind_{\delta,B} +f \Vert},
\]
the supremum being taken over all finite subsets $A$, $B$ of $\NN$ with $ |A|=|B|\le m$, all signs $\varepsilon\in\EE^{A}$ and $\delta\in\EE^{B}$, and all $f\in \XX$ with $\Vert f\Vert_{\infty}\le 1$ and $\supp(f)\cap (A\cup B)=\emptyset$.
A basis $\XB$ is \emph{symmetric for largest coefficients} (SLC for short) if $\sup_m \slc_m<\infty$. Imposing the extra assumption $A\cap B=\emptyset$ in the definition, we obtain the `disjoint' counterpart of the SLC parameters, herein denoted by $\slcd_m=\slcd_m[\XB,\XX]$.

Bases with $\sup_m \slc_m=1$ were first considered in \cite{AW2006}, where they were called bases with \emph{property (A)}. The parameters that quantify the symmetry for largest coefficients of a basis appear naturally when looking for estimates for the Lebesgue constants that are close to one (see \cites{AW2006,DOSZ2011,DKOSZ2014,AAW2018b}). In fact, if we put
\[
A_p=(2^p-1)^{1/p}, \quad 0<p\le 1,
\]
and $\XX$ is a $p$-Banach space, 
\begin{equation}\label{eq:LebUncSLC}
\max\left\{ \uncc_m, \slcd_m\right\} \le \leb_m \le A_p^2 \uncc_{2m} \slcd_m, \quad m\in\NN,
\end{equation}
(see \cite{BBG2017}*{Proposition 1.1} and \cite{AABW2019}*{Theorem 7.2}).

Thus, since $\slcd_m\le \slc_m \le (\slcd_m)^2$, in the case when $\XX$ is a Banach space, $\Ct_g=1$ if and only if $\uncc_m=\slc_m=1$ for all $m\in\NN$. Taking into account that for some constant $C$,
\[
\slc_m\le C \slcd_m,\quad m\in \NN,\]
(see Proposition~\ref{prop:SLCDisjoint} below),
equation \eqref{eq:LebUncSLC} yields
\[
\frac{1}{ C_1} \max\left\{ \unc_m, \slc_m \right\} \le \leb_m \le C_2 \unc_m \, \slc_m, \quad m\in\NN,
\]
for some constants $C_1$ and $C_2$. Thus, again, the attempt to obtain better asymptotic estimates for $(\leb_m)_{m=1}^{\infty}$ using parameters larger than $(\dem_m)_{m=1}^{\infty}$ is futile.

The vestiges of symmetry found in some bases can also be measured qualitatively by means of the upper and lower democracy functions, and using the concept of dominance between bases.

The \emph{upper super-democracy function}, also known as \emph{fundamental function}, of a basis $\XB$ of a quasi-Banach space $\XX$ is defined as
\begin{equation}\label{eq:udf}
\usdf(m)=\usdf[\XB,\XX](m)=\sup\left\lbrace \left\Vert \Ind_{\varepsilon,A} \right\Vert \colon |A|\le m,\, \varepsilon\in\EE^A \right\rbrace, \quad m\in\NN,
\end{equation}
while the \emph{lower super-democracy function} of $\XB$ is
\[
\lsdf(m)=\lsdf[\XB,\XX](m)=\inf\left\lbrace \left\Vert \Ind_{\varepsilon,A} \right\Vert \colon |A|= m,\, \varepsilon\in\EE^A \right\rbrace, \quad m\in\NN.
\]

A basis $\YB=(\yy_n)_{n=1}^\infty$ of a quasi-Banach space $\YY$ is said to \emph{dominate} a basis $\XB=(\xx_n)_{n=1}^\infty$ of a quasi-Banach space $\YY$ if there is a bounded linear map $S\colon\XX\to\YY$ such that $S(\xx_n)=\yy_n$ for all $n\in\NN$. If $\Vert S\Vert \le D$, we will say that $\YB$ $D$-dominates $\XB$.

Roughly speaking, we are interested in bases that can be `squeezed' (using domination) between two symmetric bases with equivalent fundamental functions.

\begin{definition}
A basis $\XB$ of a quasi-Banach space $\XX$ is said to be \emph{squeeze symmetric} if there are quasi-Banach spaces $\XX_1$ and $\XX_2$ with symmetric bases $\XB_1$ and $\XB_2$ respectively such that
\begin{enumerate}[label=(\roman*), leftmargin=*]
\item\label{it:sq1} $\usdf[\XB_1,\XX_1]\le \lsdf[\XB_2,\XX_2]$,
\item\label{it:sq2} $\XB_1$ dominates $\XB$, and
\item\label{it:sq3} $\XB$ dominates $\XB_2$.
\end{enumerate}
\end{definition}
This condition guarantees in a certain sense the optimality of the compression algorithms with respect to the basis (see \cite{Donoho1993}). For an approach to squeeze symmetric bases from this angle, we refer the reader to \cites{AlbiacAnsorena2016,Woj2000,CDVPX1999}. Squeeze symmetry also serves in some situations as a tool to derive other properties of the bases like being quasi-greedy for instance (see \cites{KoTe1999,DKK2003, BBG2017, BBGHO2018, AADK2019b,Woj2003}).

Although it was not originally given this name, squeeze symmetry was introduced in \cite{AABW2019} to give relief to a feature that has been implicit in greedy approximation with respect to bases since the early stages of the theory. The techniques developed in \cite{AABW2019}*{\S9} show that squeeze symmetric bases are closely related to embeddings involving Lorentz sequence spaces. Before giving a precise formulation of this connection, we recall that if a basis $\XB=(\xx_n)_{n=1}^\infty$ of $\XX$ is $1$-symmetric, i.e., \eqref{eq:symbis} holds with $C_1=C_2=1$, then $\usdf[\XB,\XX]=\lsdf[\XB,\XX]$. Note also that every symmetric basis is $1$-symmetric under a suitable renorming of the space; thus, there is no real restriction in assuming that all symmetric bases are $1$-symmetric.

Recall that a weight $\prim=(s_m)_{m=1}^\infty$ is the \emph{primitive weight} of a weight $\ww=(w_n)_{n=1}^\infty$, in which case we say that $\ww$ is the \emph{discrete derivative} of $\prim$, if $s_m=\sum_{n=1}^m w_n$ for all $m\in\NN$.

\begin{theorem}[see \cite{AABW2019}*{Equation (9.4), Lemma 9.3 and Theorem 9.12}]\label{thm:SSChar}
Let $\XX$ be a quasi-Banach space with a basis $\XB$. Let $\ww$ be the discrete derivative of the fundamental function of $\XB$.
\begin{enumerate}[label=(\alph*), leftmargin=*, widest=a]
\item $\XB$ is squeeze symmetric if and only if the coefficient transform defines a bounded linear operator from $\XX$ into $d_{1,\infty}(\ww)$.
\item Set $
\Gamma=\inf D_1 D_2$, where the infimum is taken over all $1$-symmetric bases $\XB_1$ and $\XB_2$ of quasi-Banach spaces $\XX_1$ and $\XX_2$ such that $\XB_1$ $D_1$-dominates $\XB$, $\XB$ $D_2$-dominates $\XB_2$, and
$
\usdf[\XB_1,\XX_1]\le\usdf[\XB_2,\XX_2].
$
Then there are constants $C_1$ and $C_2$ depending only on the modulus of concavity of $\XX$ such that
\[
\frac{1}{C_2} \Gamma \le \Vert \Fou\Vert_{\XX\to d_{1,\infty}(\ww)} \le C_2 \Gamma.
\]
\end{enumerate}
\end{theorem}
Theorem~\ref{thm:SSChar} serves as motivation to define a new kind of Lebesgue parameters associated to an arbitrary basis $\XB$, which will eventually be key in solving Temlyakov's problem.

For each $f\in\XX$, let $(\dr_m(f))_{m=1}^\infty$ be the non-increasing rearrangement of $|\Fou(f)|$. Note that if we put
\[
\fss_m= \fss_m[\XB,\XX]=\sup_{f\in\XX\setminus\{0\}} \frac{\dr_m(f)}{\Vert f\Vert}, \quad m\in\NN,
\]
then the norm of the coefficient transform as an operator from $\XX$ into $d_{1,\infty}(\ww)$ is $\sup_m \fss_m \usdf(m)$.
So, for $m\in\NN$ we define the \emph{$m$th squeeze symmetry parameter} of $\XB$,
\[
\sq_m=\sq_m[\XB,\XX]= \fss_m[\XB,\XX] \usdf[\XB,\XX](m).
\]

We can also approach the definition of the squeeze symmetry parameters from a functional angle. If for $A\subseteq\NN$ finite,
we denote by $\Fou_A$ the projection of the coefficient transform onto $A$,
\[
\Fou_A\colon\XX \to \FF^\NN, \quad f\mapsto \Fou(f) \chi_A,
\]
then
\[
\sq_m[\XB,\XX] = \sup \{ \Vert \Fou_A\Vert_{\XX\to d_{1,\infty}(\ww)} \colon |A|\le m\}.
\]

Using that $(\dr_m(f))_{m=1}^\infty$ is non-increasing for all $f\in\XX$ yields the properties of $(\sq_m)_{m=1}^{\infty}$ gathered in the following lemma for further reference.

\begin{lemma}\label{lem:SSParameters}
Let $\XB=(\xx_n)_{n=1}^\infty$ be a basis of a quasi-Banach space $\XX$ with coordinate functionals $(\xx_n^*)_{n=1}^\infty$, and let $m\in\NN$.
\begin{enumerate}[label=(\roman*), leftmargin=*, widest=ii]
\item\label{SSParameters:1} $1/\fss_m$ is the infimum value of $\Vert f\Vert$, where $f$ runs over all functions in $\XX$ with
\[
|\{n\in\NN\colon |\xx_n^*(f)|\ge 1\}|\ge m.
\]
\item\label{SSParameters:2}
$\sq_m[\XB,\XX]$ is the smallest constant $C$ such that $t \Vert \Ind_{\varepsilon,A} \Vert \le C \Vert f\Vert$ whenever $t\in[0,\infty)$, $f\in\XX$, $A\subseteq\NN$ and $\varepsilon\in\EE^A$, satisfy
\[
|A|\le\min\{ m,|\{n\in\NN\colon |\xx_n^*(f)|\ge t\}|\}.
\]
\item\label{SSParameters:3}
$\sq_m[\XB,\XX]$ is the smallest constant $C$ such that
\[
\min_{n\in B} |\xx_n^*(f)| \, \Vert \Ind_{\varepsilon,A} \Vert \le C \Vert f\Vert
\]
whenever $A\subseteq\NN$, $\varepsilon\in\EE^A$, and $B$ greedy set of $f\in\XX$ satisfy $|A| = |B|\le m$.
\end{enumerate}
In particular, the sequence of coefficients $(\fss_m)_{m=1}^\infty$ is non-increasing and $(\sq_m)_{m=1}^\infty$ is non-decreasing.
\end{lemma}

\subsection{Other parameters of bases associated with the TGA}
In greedy approximation with respect to bases it is also of interest to compare the error $\Vert f-\GG_{m}(f)\Vert$ with the best error in the approximation of $f$ by $m$-term coordinate projections. Thus, given a basis $\XB$ of a quasi-Banach space $\XX$, for $m\in\NN$ we put
\[
\widetilde{\sigma}_{m}[\XB, \XX](f):=\widetilde{\sigma}_{m}(f) = \inf\{ \Vert f - S_A(f) \Vert \colon |A| =m\},
\]
and define the \emph{$m$th almost greedy constant} as
\[
\leba_m=\leba_m[\XB,\XX]=\sup\left\{ \frac{\Vert f-\GG_{m}(f)\Vert}{\widetilde{\sigma}_{m}(f)} \colon f\in\XX\setminus\Sigma_m\right\}.
\]
By definition \cite{DKKT2003}, the basis $\XB$ is \emph{almost greedy} if and only if $\sup_m \leba_m<\infty$.

We say that $A$ is a \emph{greedy set} of $f\in\XX$ (relative to the basis $\XB$) if
\[
|\xx_n^*(f)|\ge |\xx_k^*(f)|, \quad n\in A,\; k\in\NN\setminus A,
\]
in which case $S_A(f)$ is called a \emph{greedy sum of order $m :=|A|$} of $f$. The greedy sums of a function $f$ need not be unique. In this regard, the thresholding greedy algorithm $(\GG_{m}(f))_{m=1}^\infty$ is a natural way to construct for each $m$ a greedy sum of $f$. Indeed, we can use the natural ordering of $\NN$ to recursively define for each $f\in\XX$ and $m\in\NN$ a greedy set $A_m(f)$ of cardinality $m$ as follows. Assuming that $A_{m-1}(f)$ is defined we set
\[
\textstyle
k(A,m)=\min\{ k \in \NN\setminus A_{m-1}(f) \colon |\xx_k^*(f)|= \max_{n\notin A_{m-1}(f)} |\xx_n^*(f)| \},
\]
and $A_m(f)=A_{m-1}(f)\cup \{ k(A,m)\}$. With this agreement, we have
\[
\GG_{m}[\XB, \XX](f):=\GG_{m}(f) = S_{A_m(f)}(f), \quad f\in \XX, \; m\in\NN.
\]

From the continuity of the quasi-norm and a standard perturbation technique we infer that the Lebesgue constant $\leb_m$ is the smallest constant $C$ such that
\[
\Vert f-S_A(f)\Vert \le C \sigma_m(f), \quad f\in\XX,\; \text{$A$ greedy set of $f$},\; |A|\le m,
\]
and the almost greedy constant $\leba_m$ is the optimal constant $C$ such that
\begin{equation}\label{eq:AGImproved}
\Vert f-S_A(f)\Vert \le C \widetilde{\sigma}_k(f), \; \text{$A$ greedy set of $f$},\; k\le |A|\le m,
\end{equation}
(see \cite{AlbiacAnsorena2017b}*{Lemma 2.2} and \cite{AABW2019}*{Lemma 6.1}). Consequently, the sequences $(\leb_m)_{m=1}^\infty$ and $(\leba_m)_{m=1}^\infty$ are non-increasing.

The \emph{$m$th quasi-greedy constant} $\qg_m$ and its complemented counterpart $\qgc_m$ are defined by
\[
\qg_m=\sup_{1\le k \le m} \Vert \GG_{k}\Vert,
\quad \qgc_m=\sup_{1\le k \le m} \Vert \Id_\XX- \GG_{k}\Vert.
\]
Similarly to the unconditionality parameters, in the case when $\XX$ is a $p$-Banach space, these parameters are related by the inequalities
\[
(\qg_m)^p\le 1+(\qgc_m)^p
\;
\text{and}\;(\qgc_m)^p\le 1+(\qg_m)^p.
\]
The same perturbation technique used before gives that $\qgc_m$ is the smallest constant $C$ such that
\begin{equation*}
\left\Vert f-S_A(f)\right\Vert \le C \left\Vert f \right\Vert,\; \text{$A$ greedy set of $f$},\; |A|\le m,
\end{equation*}
and the quasi-greedy constant $\qg_m$ is the smallest constant $C$ such that
\begin{equation}\label{eq:BDProj}
\left\Vert S_A(f)\right\Vert \le C \left\Vert f \right\Vert,\; \text{$A$ greedy set of $f$},\; |A|\le m.
\end{equation}

If we impose \eqref{eq:BDProj} only to functions $f$ such that $|\Fou(f)|$ is constant on $A$, we will denote by $\qglc_m$ the corresponding parameter and will call it the \emph{$m$th quasi-greedy parameter for largest coefficients}, or $m$th QGLC parameter for short. Finally, the \emph{$m$th uncoditionality parameter for constant coefficients}, or $m$th UCC parameter for short, denoted $\ucc_m$, is defined by imposing condition \eqref{eq:BDProj} only to functions $f$ with $|\Fou(f)|$ constant and $|\supp(f)|\le m$. A basis is quasi-greedy for largest coefficients (resp.\ uncoditional for constant coefficients) if the corresponding sequence of parameters is uniformly bounded.

Note that the parameters $(\dem_m)_{m=1}^\infty$, $(\sdem_m)_{m=1}^\infty$, $(\slc_m)_{m=1}^\infty$, $(\slcd_m)_{m=1}^\infty$, $(\qg_m)_{m=1}^\infty$, $(\qgc_m)_{m=1}^\infty$, $(\qglc_m)_{m=1}^\infty$, and $(\ucc_m)_{m=1}^\infty$, as well as the function $\usdf$ are non-decreasing by definition. Again, standard approximation arguments give that we can equivalently define all the above parameters by restricting us to finitely supported functions. Using this it is easy to infer that $(\unc_m)_{m=1}^\infty$, and $(\uncc_m)_{m=1}^\infty$ also are non-decreasing.

\subsection{Organization of the article.}
With all the previous ingredients we are now in a position to describe in some detail the contents of this paper. After setting the most heavily used terminology and notation in Section~\ref{term}, in Section~\ref{sect:Main} we prove the central result of this paper:
\begin{theorem}\label{thm:main}
Let $\XB$ be a basis of a quasi-Banach space $\XX$. There are constants $C_1$ and $C_2$ (depending only on the modulus of concavity of $\XX$) such that
\[
\frac{1}{C_1} \leb_m[\XB,\XX] \le \max\{\sq_m[\XB,\XX], \unc_m[\XB,\XX]\}\le C_2 \, \leb_m[\XB,\XX], \quad m\in\NN.
\]
\end{theorem}

In Section~\ref{sect:Main} we will also show the almost greedy analogue of Theorem~\ref{thm:main}, which in this case involves the almost greedy constants, the quasi-greedy parameters and some relatives of the squeeze symmetry parameters (see Theorem~\ref{thm:MainNotMain}). In Sections~\ref{sect:Applications}, \ref{sect:SQUnc}, and \ref{Applications:Duality}, we investigate the theoretical applications of the new Lebesgue parameters and the new Lebesgue-type estimates derived from them. In particular, we compare Theorem~\ref{thm:main} with other bounds for the Lebesgue constants that can be found in the literature. As a matter of fact, as we will see, most known estimates for the Lebesgue parameters can be deduced from Theorem~\ref{thm:main}. For instance, Theorem~\ref{GHOthm} can be deduced from Theorem~\ref{thm:GHOI}. Since the practical interest of Theorem~\ref{thm:main} depends on the ability to estimate the squeeze symmetry parameters, we devote Section~\ref{sect:DLP} to relating these parameters with other parameters that quantify different degrees of democracy. In Section~\ref{examplesdeleixample} we compute the Lebesgue parameters and obtain Lebesgue-type estimates in some important examples.

\section{Terminology and Notation}\label{term}\noindent 
Throughout this paper we use standard facts and notation from Banach spaces and approximation theory (see e.g.\ \cite{AlbiacKalton2016}). The reader will find the required specialized background and notation on greedy-like bases in quasi-Banach spaces in the recent article \cite{AABW2019}; however a few remarks are in order.

Let us first recall that a quasi-Banach space is a vector space $\XX$ over the real or complex field $\FF$ equipped with a map $\|\cdot\|\colon X\to [0,\infty)$, called a \emph{quasi-norm}, which satisfies all the usual properties of the norm with the exception that the triangle law is replaced with the inequality
\begin{equation}\label{defquasinorm}
\|f+g\|\leq \kappa( \| f\| + \|g\|),\quad f,g\in X,
\end{equation}
for some $\kappa\ge 1$ independent of $f$ and $g$, and moreover $(X,\|\cdot\|)$ is complete. The \emph{modulus of concavity} of the quasi-norm is the smallest constant $\kappa\ge 1$ in \eqref{defquasinorm}. Given $0<p\le 1$, a \emph{$p$-Banach space} will be a quasi-Banach space whose quasi-norm is $p$-subadditive, i.e.,
\[
\Vert f+g\Vert^p \le \Vert f\Vert^p +\Vert g \Vert^p, \quad f,g\in\XX.
\]
Any $p$-Banach space has modulus of concavity at most $2^{1/p-1}$. Conversely, by the Aoki-Rolewicz Theorem \cites{Aoki1942,Rolewicz1957}, any quasi-Banach space with modulus of concavity at most $2^{1/p-1}$ is $p$-Banach under an equivalent quasi-norm. So, we will suppose that all quasi-Banach spaces are $p$-Banach spaces for some $0<p\le 1$. As a consequence of this assumption, all quasi-norms will be continuous.

The linear space of all eventually null sequences will be denoted by $c_{00}$, and $c_0$ will be the Banach space consisting of all null sequences. Let $0<q\le \infty$ and $\ww=(w_n)_{n=1}^\infty$ be a \emph{weight}, i.e., a sequence of non-negative scalars with $w_1>0$. Let $(s_m)_{m=1}^\infty$ be the primitive weight of $\ww$. We will denote by $d_{1,q}(\ww)$ the Lorentz sequence space consisting of all $f\in c_0$ whose non-increasing rearrangement $(a_n)_{n=1}^\infty$ satisfies
\[
\Vert f\Vert_{d_{1,q}(\ww)}=\left( \sum_{n=1}^\infty a_n^q s_n^{q-1}w_n\right)^{1/q}<\infty,
\]
with the usual modification if $q=\infty$. If $\prim=(s_m)_{m=1}^\infty$ is \emph{doubling}, i.e., there exists a constant $C$ such that $s_{2m} \le C s_m$ for all $m \in \NN$, $d_{1,q}(\ww)$ is a quasi-Banach space. Moreover, if $\prim$ is doubling, the unit vector system is a $1$-symmetric basis of the its closed linear span, and, if $q<\infty$, the aforementioned closed linear span is $d_{1,q}(\ww)$. Regardless of the value of $q$, the fundamental function of the unit vector system of $d_{1,q}(\ww)$ is of the same order as $(s_m)_{m=1}^\infty$. We refer the reader to \cite{AABW2019}*{\S9.2} for a concise introduction to Lorentz sequence spaces.

The norm of an operator $T$ from a quasi-Banach space $\XX$ into a quasi-Banach space $\YY$ will be denoted by $\Vert T\Vert_{\XX\to\YY}$ or simply $\Vert T\Vert$ if the spaces are clear from context. A sequence $(\xx_n)_{n=1}^\infty$ in a quasi-Banach space $\XX$ is said to be \emph{semi-normalized} if
\[
0 < \inf_{n\in\NN} \Vert \xx_n\Vert \le \sup_{n\in\NN} \Vert \xx_n\Vert<\infty.
\]

We will only deal with biorthogonal systems $(\xx_n,\xx_n^*)_{n=1}^\infty$ for which the democracy and the unconditionality parameters of $\XB=(\xx_n)_{n=1}^\infty$ are finite. Note that
\[
\unc_1[\XB,\XX]=\sup_{n\in\NN} \Vert \xx_n\Vert \, \Vert \xx_n^*\Vert,
\]
and that, if $\unc_1<\infty$, then $\unc_m<\infty$ for all $m\in\NN$. In turn, taking into account that
\begin{equation*}
\dem_1[\XB,\XX] =\frac{\sup_n \Vert \xx_n\Vert}{\inf_n \Vert \xx_n\Vert},
\end{equation*}
we infer that, if $\dem_1$ and $\unc_1$ are both finite, then $\dem_m<\infty$ for all $m\in\NN$. Indeed, if $\XX$ is a $p$-Banach space and $|A|=|B|\le m$,
\[
\frac{\Vert \Ind_A\Vert }{\Vert \Ind_B\Vert }
\le m^{1/p} \frac{ \sup_{n\in A} \Vert \xx_n\Vert}{\inf_{n\in B} \Vert \xx_n\Vert} \inf_{n\in B} \frac{\Vert \xx_n\Vert\, \Vert \xx_n^*\Vert }{ \Vert \Ind_B \Vert\, \Vert \xx_n^* \Vert}
\le m^{1/p} \dem_1 \unc_1.
\]
Finally, we note that $\max\{\dem_1,\unc_1\}<\infty$ if and only if
\begin{equation}\label{eq:finiteparameters}
C[\XB]:=\sup_{n\in\NN} \max\{\Vert \xx_n\Vert, \Vert \xx_n^*\Vert \}<\infty,
\end{equation}
in which case both $\XB$ and $\XB^*=(\xx_n^*)_{n=1}^\infty$ are semi-normalized.

Given a basis $\XB=(\xx_n)_{n=1}^\infty$ of a quasi-Banach space $\XX$, $f\in \XX$ and $f^*\in\XX^*$ we define sequences $\varepsilon(f)$ and $\varepsilon(f^*)$ in $\EE^\NN$ by
\begin{align*}
\varepsilon(f)&=(\varepsilon_n(f))_{n=1}^\infty=( \sgn (\xx_n^*(f)))_{n=1}^\infty,\\
\varepsilon(f^*)&=(\varepsilon_n(f^*))_{n=1}^\infty=( \sgn (f^*(\xx_n))_{n=1}^\infty,
\end{align*}
where $\sgn(0)=1$ and $\sgn(a)=a/|a|$ if $a\not=0$.

Given two sequences of parameters $(\alpha_m)_{m=1}^\infty$ and $(\beta_m)_{m=1}^\infty$ related to bases of quasi-Banach spaces, the symbol $\alpha_m\lesssim C\beta_m$ will mean that for every $0<p\le 1$ there is a constant $C$ such that $\alpha_m[\XB,\XX] \le C \beta_m[\XB,\XX]$ for all $m\in\NN$ and for all bases $\XB$ of a $p$-Banach space $\XX$.

\section{Lower and upper bounds for the Lebesgue constants}\label{sect:Main}\noindent
Konyagin and Temlyakov's characterization of greedy bases in terms of unconditionality and democracy can be reformulated quantitatively by saying that the Lebesgue constants $(\leb_m)_{m=1}^{\infty}$ for the greedy algorithm are bounded above and below by parameters $(\unc_m)_{m=1}^{\infty}$ and $(\dem_m)_{m=1}^{\infty}$ that measure the unconditionality and the democracy of the basis. Since the dependence in \eqref{eq:UncG} is not linear, a natural question raised by Temlyakov \cite{Temlyakov2011}, is to find parameters related to the unconditonality and the democracy of the basis whose maximum value grows as the Lebesgue constant.

Theorem~\ref{thm:main} provides a satisfactory answer to this question by means of the squeeze symmetry parameters. We get started with a lemma which connects some important constants that we will need by using the $p$-convexitity techniques developed in \cite{AABW2019}*{\S2}. The symbol $\FieldC$ stands for $2$ if $\FF=\RR$, and for $4$ if $\FF=\CC$.
\begin{lemma}\label{lem:SignVsChar}
Let $0<p\le 1$, and let $\XB$ be a basis of a $p$-Banach space $\XX$. Given $A\subseteq\NN$ finite, $f\in\XX$, and $\delta=(\delta_n)_{n\in A}\in\EE^A$, we denote
\begin{align*}
K[A,f]&=\textstyle \sup \{ \Vert f+ \sum_{n\in A} a_n\,\xx_n\Vert \colon |a_n|\le 1 \},\\
L[A,f]&=\sup \{ \Vert \Ind_{\varepsilon,A}+f\Vert \colon \varepsilon\in \EE^A \},\\
M[\delta,A,f]&=\textstyle \sup \{ \Vert f+ \sum_{n\in A} a_n\, \delta_n\,\xx_n\Vert \colon 0\le a_n\le 1 \},\quad \text{ and }\\
N[\delta,A,f]&=\sup \{ \Vert \Ind_{\delta,B}+f\Vert \colon B \subseteq A \}.
\end{align*}
Set $C_p=\FieldC^{1/p}$ if $f=0$ and $C_p=(1+\FieldC)^{1/p}$ otherwise.
Then:
\begin{align*}
K[A,f]&\le \min\{C_p\, M[\delta,A,f], A_p\, L[A,f]\} \text{ and }\\
M[\delta,A,f]&\le A_p\, N[\delta,A,f].
\end{align*}
\end{lemma}

\begin{proof}
In the case when $\FF=\CC$, set $\gamma_j=i^j$ for $j=1$, $2$, $3$, $4$. In the case when $\FF=\RR$, set $\gamma_j=(-1)^j$ for $j=1$, $2$. Given $(a_n)_{n\in A}\in\FF^A$ with $|a_n|\le 1$, there are $(a_{j,n})_{n\in A}\in\FF^A$, $j\in \NN$, $1\le j\le\FieldC$, in $[0,1]$ such that $a_n=\sum_{j=1}^4 \gamma_j a_{j,n}$. The identity
\[
g:= f+ \sum_{n\in A} a_n\,\xx_n=f+ \sum_{j=1}^\FieldC \gamma_j \left( f+ \sum_{n\in A} a_{j,n}\,\xx_n\right)
\]
gives $\Vert g \Vert \le C_p M[\delta,A,f]$. The other inequalities follow readily from \cite{AABW2019}*{Corollary 2.3}.
\end{proof}

\begin{proof}[Proof of Theorem~\ref{thm:main}] Let us first note that by \eqref{kmkmc}, in the proof of the theorem we can replace $\unc_m$ with $\uncc_m$. The inequality
\begin{equation}\label{U:G}
\uncc_m \le \leb_m, \quad m\in\NN,
\end{equation}
can be deduced from the fact that, given $f\in\XX$ and $A\subseteq \NN$ finite, there is $h\in \spn(\xx_n \colon n\in A)$ such that $A$ is a greedy set of $f+h$ (see \cite{GHO2013}). Thus to complete the proof it suffices to show that 
\begin{equation}\label{eq.former}
\leb_m\le C_1 \max\{\sq_m,\uncc_m\}
\end{equation}
and
\begin{equation}\label{eq.later}
\sq_m\le C_2 \max\{\uncc_m,\leb_m\}
\end{equation}
for all $m\in\NN$ and some constants $C_1$ and $C_2$.

To show \eqref{eq.former}, assume that $\XX$ is a $p$-Banach space, $0<p\le 1$. Let $A$ be a greedy set of $f\in\XX$ with $|A|=m<\infty$, and pick $z=\sum_{n\in B}a_n\, \xx_n$ with $\vert B\vert=|A|$. Notice that
\begin{equation*}\label{c1-g}
\max_{n\in B\setminus A}\vert\xx_n^*(f)\vert\le\min_{n\in A\setminus B}\vert\xx_n^*(f)\vert
=\min_{n\in A\setminus B}\vert\xx_n^*(f-z)\vert.
\end{equation*}
Set $k=|B\setminus A|=|A\setminus B|$. On one hand, by Lemma~\ref{lem:SSParameters}~\ref{SSParameters:2} and Lemma~\ref{lem:SignVsChar},
\[
\Vert S_{B\setminus A}(f)\Vert \le A_p \sq_m \Vert f-z\Vert.
\]
On the other hand, since $|A\cup B|=m+k$,
\[
\Vert (f-z) -S_{A\cup B)}(f-z)\Vert\le \uncc_{m+k} \Vert f -z\Vert\le \uncc_{2m}\Vert f -z\Vert.
\]
Since
\[
f- S_A(f) = (f-z)-S_{A\cup B}(f-z) + S_{B\setminus A}(f),
\]
combining both inequalities gives
\begin{align*}
\Vert f- S_A(f)\Vert^p
&\le \left((\uncc_{2m})^p + (A_p \sq_m)^p\right)\Vert f\Vert^p\\
&\le \left( 1+2 (\unc_m)^p + (A_p \sq_m)^p\right)\Vert f\Vert^p.
\end{align*}

Let us now prove inequality \eqref{eq.later}. In order to apply Lemma~\ref{lem:SSParameters}~\ref{SSParameters:3}, we pick $A\subseteq\NN$, $f\in\XX$, and $B$ greedy set of $f$ with $|A|=|B|\le m$. Set $t=\min_{n\in B} |\xx_n^*(f)|$ and pick $(a_n)_{n\in A\cup B}$ in $[0,t]$. Let us put
\[
y=\sum_{n\in B} a_n \, \varepsilon_n(f)\, \xx_n, \quad
z=\sum_{n\in A\setminus B} a_n \, \varepsilon_n(f)\, \xx_n
\]
and $g=f-z$. On the one hand, since $|A\setminus B|\le |B|$, $B$ is a greedy set of $g$, and $g-S_B(g)=f-S_B(f)-z$, we have
\[
\Vert f-S_B(f)-z\Vert\le\leb_m \Vert g+z\Vert=\leb_m \Vert f\Vert.
\]
On the other hand, since $|\xx_n^*(f-y)|\le \xx_n^*(f)|$ and $\sgn(\xx_n^*(f-y))=\sgn(\xx_n^*(f))$ for all $n\in B$, applying \cite{AABW2019}*{Corollary 2.3} yields
\[
\left\Vert f-S_{B}(f) +y \right\Vert
=\left\Vert f-S_{B}(f-y) \right\Vert
\le A_p \uncc_m \Vert f \Vert.
\]
Combining both estimates we obtain
\[
\left\Vert y-z \right\Vert \le \left((A_p\, (\uncc_m)^p + (\leb_m)^p\right)^{1/p} \Vert f \Vert.
\]
Hence, by Lemma~\ref{lem:SignVsChar},
\[
\left\Vert \sum_{n\in A\cup B} b_n\, \xx_n \right\Vert \le \FieldC^{1/p} \left( (A_p\, \uncc_m)^p + (\leb_m)^p\right)^{1/p}\Vert f\Vert.
\]
Applying this inequality with $b_n=0$ for all $n\in B\setminus A$ and $|b_n|=1$ for all $n\in A$, we are done.
\end{proof}

Since a basis is almost greedy if and only if it is quasi-greedy and democratic (see \cite{DKKT2003}*{Theorem 3.3} and \cite{AABW2019}*{Theorem 6.3}), it seems natural to look for democracy-like parameters which, when combined with the quasi-greeediness parameters, provide optimal bounds for the growth of the almost greedy constants. For this purpose, we define the \emph{disjoint squeeze symmetry parameter} $\sqd_m=\sqd_m[\XB,\XX]$ as the smallest constant $C$ such that
\[
\min_{n\in B} |\xx_n^*(f)| \, \Vert \Ind_{\varepsilon,A} \Vert \le C \Vert f\Vert
\]
whenever $A\subseteq\NN$, $\varepsilon\in\EE^A$, and $B$ is a greedy set of $f\in\XX$ satisfy $A\cap\supp(f)=\emptyset$ and $|A| = |B|\le m$. The following estimates imply that a basis is squeeze symmetric if and only if $\sup_m \sqd_m<\infty$.



\begin{lemma}\label{lem:SLCSSD}
Let $\XB$ be a basis of a quasi-Banach space $\XX$. Then
\[
\sqd_m[\XB,\XX] \le \sq_m[\XB,\XX] \le (\sqd_m[\XB,\XX])^2.
\]
\end{lemma}

\begin{proof}
Lemma~\ref{lem:SSParameters}~\ref{SSParameters:3} yields the left hand-side inequality. Let $A$ and $B$ be disjoint sets with $|A|=|B|=m$. Then
\[
\Vert \Ind_{\varepsilon,A}\Vert \le \sqd_m\Vert \Ind_{\delta,B}\Vert
\]
for all $\varepsilon\in\EE^A$ and $\delta\in\EE^B$. Since we can restrict ourselves to finitely supported vectors, combining this fact with Lemma~\ref{lem:SSParameters}~\ref{SSParameters:3} yields the right hand-side inequality.
\end{proof}

\begin{theorem}\label{thm:MainNotMain}
Let $\XB$ be a basis of a quasi-Banach space $\XX$. There are constants $C_1$ and $C_2$ depending only on the modulus of concavity of $\XX$ such that
\[
\frac{1}{C_1} \leba_m[\XB,\XX] \le \max\{\sqd_m[\XB,\XX], \qg_m[\XB,\XX]\}\le C_2 \leba_m[\XB,\XX], \quad m\in\NN.
\]
\end{theorem}

\begin{proof}
Assume that $\XX$ is a $p$-Banach space, $0<p\le 1$. Inequality \eqref{eq:AGImproved} yields $\qgc_m \le \leba_m$ (see \cite{BBG2017}*{Proposition 1.1}). To conclude the proof of the right-side estimate, we pick $A\subseteq\NN$, $\varepsilon\in\EE^A$, $f\in\XX$, and $B$ greedy set of $f$ with $|A|=|B|\le m$ and $A\cap \supp(f)=\emptyset$. Set $t=\min_{n\in B} |\xx_n^*(f)|$. Taking into account that $B$ is a greedy set of $g:=f+ t \Ind_{\varepsilon,A}$ and that $t\Ind_{\varepsilon, A}= g-S_B(g) - (f-S_B(f))$ we obtain
\begin{align*}
t \Vert \Ind_{\varepsilon,A} \Vert^p
&\le \Vert g- S_B(f)\Vert^p + \Vert f - S_B(f)\Vert^p \\
&\le (\leba_m)^p \Vert g- S_A(g)\Vert^p + (\qgc_m)^p \Vert f \Vert^p\\
&= ((\leba_m)^p+(\qgc_m)^p)\Vert f\Vert^p.
\end{align*}

To prove the left-side estimate, we pick a greedy set $A$ of $f\in\XX$ and $B\subseteq \NN$ with $|A|=|B|= m$. Notice that $|A\setminus B|=|B\setminus A|\le m$, that $A\setminus B$ is a greedy set of $g:=f-S_B(f)$, that
\[
g-S_{A\setminus B}(g)=f-S_{A\cup B}(f)\; \text{ and }\; f -S_A(f)= f-S_{A\cup B}(f) + S_{B\setminus A}(f),
\]
and that
\[
\max_{n\in A\setminus B} |\xx_n^*(f)|\le \min_{n\in B\setminus A} \xx_n^*(g)|.
\]
Hence, applying Lemma~\ref{lem:SignVsChar} we obtain
\[
\Vert f -S_A(f) \Vert^p = \Vert f-S_{A\cup B}(f)\Vert^p + \Vert S_{B\setminus A}(f) \Vert^p\le C^p \Vert g\Vert^p,
\]
where $C^p=(\qgc_m)^p + (A_p \sqd_m)^p$. Consequently, $\leba_m \le C$.
\end{proof}

\section{Lebesgue-type inequalities for truncation quasi-greedy bases}\label{sect:Applications}\noindent
While, a priori, democracy and unconditionality are independent properties of each other, in practice, no matter how we upgrade the democracy of a basis, it ends up having some traits in common with unconditionality. For instance, super-democratic bases share with unconditional bases the feature of being unconditional for constant coefficients (UCC for short); and, the other way around, a basis is super-democratic if and only if it is simultaneously democratic and UCC (see \cite{DKK2003}). Similarly, a basis is SLC if and only if it is democratic and QGLC (see \cite{AABW2019}*{Proposition 5.3}). In turn, quasi-greediness for constant coefficients is a weakened form of quasi-greediness, which can be considered as a weakened form of unconditionality. Finally, almost greediness can be considered a `very demanding' form of democracy, and the class of almost greedy bases overlaps with the class of unconditional bases in the class of quasi-greedy bases.

The overlapping between squeeze symmetry and unconditionality can also be identified. For that, let us first introduce the corresponding unconditionality-like property.

Given a basis $\XB$ is a quasi-Banach space $\XX$, and $A\subseteq\NN$ finite, we consider the non-linear operator
\[
\RTO_A(f)=\RTO_A[\XB,\XX](f) = \min_{n\in A} |\xx_n^*(f)| \Ind_{\varepsilon(f),A}, \quad f\in\XX.
\]
Now, for $m\in\NN$, we define the \emph{$m$th restricted truncation operator} of the basis $\XB$ as
\[
\RTO_m\colon\XX\to \XX, \quad f\mapsto \RTO_{A_m(f)}(f),
\]
and the \emph{$m$th-truncation quasi-greedy parameter} as
\[
\rto_m=\rto_m[\XB,\XX] =\sup_{1\le k\le m} \Vert \RTO_k\Vert.
\]
Those bases for which the restricted truncations operators are uniformly bounded, i.e., $\sup_m \rto_m<\infty$, will be called \emph{truncation quasi-greedy}. A standard approximation argument gives
\[
\rto_m=\sup\left\{ \frac{ \Vert \RTO_A(f) \Vert}{\Vert f \Vert } \colon A \text{ greedy set of } f\in\XX\setminus\{0\},\; |A|\le m\right\}.
\]
Thus, truncation quasi-greedy bases are QGLC. Quantitatively,
\[
\qglc_m[\XB,\XX] \le \rto_m[\XB,\XX], \quad m\in\NN.
\]
In turn, quasi-greedy basis are truncation quasi-greedy. This result, and the quantitative estimates associated with it, deserve a more detailed explanation. If $\XX$ is a Banach space, for $m\in\NN$,
\begin{equation}\label{eq:RTOQG:1}
\rto_m[\XB,\XX]\le \qg_m[\XB,\XX],
\end{equation}
(see the proofs of \cite{Woj2000}*{Theorem 3}, \cite{DKKT2003}*{Lemma 2.2} or \cite{AlbiacAnsorena2017b}*{Theorem 2.4}). However, the argument that shows inequality \eqref{eq:RTOQG:1} does not transfer to nonlocally convex quasi-Banach spaces. The authors circumvented in \cite{AABW2019} the use of convexity at the cost of getting worse estimates. In fact, if $\XX$ is a $p$-Banach space, $0<p<1$, the proof of \cite{AABW2019}*{Theorem 4.8} gives
\[
\rto_m\le \qg_m \eta_p(\qg_m),\quad fm\in\NN,
\]
where $\eta_p$ is the function defined in \cite{AABW2019}*{Equation (4.5)}. Hence (see \cite{AABW2019}*{Remark 4.9}) there is a constant $C$ (independent of $p$) such that
\begin{equation}\label{eq:RTOQG:p}
\rto_m[\XB,\XX]\le C (\qg_m[\XB,\XX])^{1+1/p},\; m\in\NN,\; 0<p<1,\; \text{$\XX$ $p$-Banach}.
\end{equation}

With an eye to relating the truncation quasi-greedy and democracy parameters with the squeeze symmetric parameters, we write down a general lemma that will be useful in applications.

\begin{lemma}\label{lem:SSDemRTO}
Let $0<p,q\le 1$. Assume that a basis $\YB$ of a $q$-Banach space $\YY$ $C$-dominates a basis $\XB$ of a $p$-Banach space $\XX$. Then
\[
\sq_m[\XB,\XX] \le \FieldC^{1/p+1/q} A_p A_q C \theta_m \rto_m[\YB,\YY] , \quad m\in\NN,
\]
where
\[
\theta_m=\sup_{1\le |A|=|B| \le m} \frac{\Vert \Ind_A[\XB,\XX]\Vert}{\Vert \Ind_B[\YB,\YY]\Vert}.
\]
\end{lemma}

\begin{proof}
Let $A\subseteq\NN$, $\varepsilon\in\EE^A$, $f\in\XX$ finitely supported, and $B$ a greedy set of $f$ with $|A| = |B|\le m$. Set $t=\min_{n\in B} |\xx_n^*(f)|$. Applying Lemma~\ref{lem:SignVsChar} we obtain
\[
\left\Vert\sum_{n\in B} a_n \, \yy_n \right\Vert\le \FieldC^{1/q} A_q \rto_m[\YB,\YY] \Vert S(f) \Vert, \quad |a_n|\le t.
\]
Hence,
\[
t \Vert \Ind_{E}[\XB,\XX] \Vert\le \FieldC^{1/q} A_q C \theta_m \rto_m[\YB,\YY] \Vert f \Vert, \quad E\subseteq A.
\]
We get the desired inequality by applying again Lemma~\ref{lem:SignVsChar}.
\end{proof}

The quantitative estimates we obtain in Proposition~\ref{prop:SSDemRTO} imply that a basis is squeeze symmetric if and only if it is truncation quasi-greedy and democratic (see \cite{AABW2019}*{Proposition 9.4 and Corollary 9.15}).

\begin{proposition}\label{prop:SSDemRTO}
Let $\XB$ be a basis of a $p$-Banach space, $0<p\le 1$. Then, for all $m\in\NN$,
\begin{equation}\label{eq:SSDemRTO}
\dem_m \le \sqd_m \; \text{ and }\; \rto_m \le \sq_m \le \FieldC^{2/p} A_p^2 \rto_m \dem_m.
\end{equation}
\end{proposition}
\begin{proof}
Using Lemma~\ref{lem:SSDemRTO} with $\XB=\YB$ and $\XX=\YY$ yields the right hand-side of the second inequality. The other two inequalities are straightforward.
\end{proof}

Proposition~\ref{prop:SSDemRTO} yields in particular that the squeeze symmetry parameters and the democracy parameters of truncation quasi-greedy bases are of the same order. Thus combining that with Theorem~\ref{thm:main} gives the following improvement of Theorem~\ref{GHOthm}.

\begin{theorem}\label{thm:GHOI}
Let $\XB$ be an truncation quasi-greedy basis of a quasi-Banach space $\XX$. There are constants $C_1$ and $C_2$ depending on the modulus of concavity of $\XX$ and the truncation quasi-greedy constant of $\XB$ such that
\[
\frac{1}{C_1}\leb_m \le \max\{\unc_m,\dem_m\} \le C_2 \leb_m, \quad, m\in\NN.
\]
\end{theorem}

We close this section with the almost greedy counterpart of Theorem~\ref{thm:GHOI}.
\begin{theorem}\label{thm:GHOAGI}
Let $\XB$ be a truncation quasi-greedy basis of a quasi-Banach space $\XX$. There are constants $C_1$ and $C_2$ depending on the modulus of concavity of $\XX$ and the truncation quasi-greedy constant of $\XB$ such that
\[
\frac{1}{C_1}\leba_m \le \max\{\qg_m,\dem_m\} \le C_2 \leba_m, \quad m\in\NN.
\]
\end{theorem}
\begin{proof}
It follows by combining Theorem~\ref{thm:MainNotMain} with Proposition~\ref{prop:SSDemRTO}.
\end{proof}

\section{Squeeze symmetry vs.\ unconditionality}\label{sect:SQUnc}\noindent
Proposition~\ref{prop:SSDemRTO} shows that squeeze symmetry and unconditionality are intertwined. This overlapping can be regarded from a different angle since squeezing a basis $\XB$ between two symmetric bases yields estimates for the unconditionality parameters of $\XB$. To give a precise formulation of this analysis we introduce some additional terminology.

Given two bases $\XB=(\xx_n)_{n=1}^\infty$ and $\YB=(\yy_n)_{n=1}^\infty$ of quasi-Banach spaces $\XX$ and $\YY$, $\dom_m[\XB, \YB]$ will denote for each $m\in \NN$ the smallest constant $C$ such that
\[
\left\Vert\sum_{n\in A} \yy_n^*(f) \, \xx_n\right\Vert \le C \Vert f \Vert, \quad |A|=m, \; f\in\YY.
\]
Notice that $\unc_m[\XB,\XX]=\dom_m[\XB,\XB]$.

Thanks to these parameters we can give an alternative reinterpretation of the fundamental function of a basis $\XB$. In fact, if $\XX$ is a $p$-Banach space, by Lemma~\ref{lem:SignVsChar},
\begin{equation}\label{eq:usdfdoml00}
\usdf[\XB,\XX](m) \le \dom_m[\XB,\ell_\infty] \le A_p\, \usdf[\XB,\XX](m), \quad m\in\NN.
\end{equation}
Here and subsequently, whenever the unit vector system $\BB=(\ee_{n})_{n=1}^{\infty}$ is a basis of a quasi-Banach space $\YY$, we will write $\dom_m[\XB,\YY]$ instead of $\dom_m[\XB,\BB]$; we will proceed analogously when $\XB$ is the unit vector system of $\XX$.

In the case when the bases $\XB$ and $\YB$ are $1$-symmetric or, more generally, $1$-subsymmetric (see \cite{Ansorena2018}), $\dom_m[\XB, \YB]$ is the smallest constant $C$ such that
\[
\left\Vert\sum_{n=1}^m a_n\, \yy_n\right\Vert \le C \left\Vert\sum_{n=1}^m a_n\, \xx_n\right\Vert, \quad a_n\ge 0.
\]

\begin{lemma}\label{lem:domunclocal}
Let $\XB$, $\XB_1$ and $\XB_2$ be bases of a quasi-Banach spaces $\XX$, $\XX_1$, and $\XX_2$ respectively. Suppose that $\XB=(\xx_n)_{n=1}^\infty$ $D$-dominates $\XB_2$. For each $A\subseteq\NN$ finite, let $T_A\colon \XX_1 \to \XX$ be the operator given by
\[
T_A(f)=\sum_{n\in A} a_n\, \xx_n, \quad (a_n)_{n=1}^\infty=\Fou[\XB_1,\XX_1](f).
\]
Set $\zeta_m=\sup_{|A|\le m} \Vert T_A \Vert$. Then, for $m\in\NN$,
\begin{align*}
\unc_m[\XB,\XX] & \le D \zeta_m \dom_m[\XB_1,\XB_2], \text{ and}\\
\sq_m[\XB,\XX]& \le D \zeta_m \sq_m[\XB_1,\XX_1] \, \dom_m[\XB_1,\XB_2].
\end{align*}
\end{lemma}

\begin{proof}
We will only prove the second inequality because it is more general. Given $f\in\XX$, let $g\in\XX_1$ be such that $\Fou(g)=\Fou(f)\chi_A$, and let $h\in\XX_2$ be such that $\Fou(h)=\Fou(f)$.
We have
\[
\Vert g \Vert
\le \dom_m[\XB_1,\XB_2]\Vert h \Vert, \quad \Vert h \Vert
\le D \dom_m[\XB_1,\XB_2] \Vert f \Vert
\]
and, since $S_A(f)=T_A(g)$,
\[
\Vert S_A(f)\Vert\le \zeta_m \Vert g \Vert.\]
Finally, if $A$ is a greedy set of $f$ and $|B|=|A|$,
\[
\Vert \Ind_A[\XB,\XX]\Vert \le \zeta_m \Vert \Ind_A[\XB_1,\XX_1]\Vert\le \sq_m [\XB_1,\XX_1] \, \Vert T_A(g) \Vert.\qedhere
\]
\end{proof}

Despite the fact that we stated Lemma~\ref{lem:domunclocal} in wide generality, in practice we will only apply it in the case when $\XB_1$ and $\XB_2$ are $1$-symmetric, so that $\sq_m[\XB_1,\XX_1]=1$ so that the parameters $\dom_m$ are easy to compute. We also point out that the best-case scenario occurs when $\XB_1$ dominates $\XB$, so that $\sup_m \zeta_m<\infty$. With an eye to applying Lemma~\ref{lem:domunclocal} to estimating Lebesgue constants, we record the parameters $\dom_m$ in some important situations.
\begin{equation}\label{eq:domlplq}
\dom_m[\ell_p,\ell_q]=m^{1/p-1/q}, \quad m\in\NN, \; p\le q.
\end{equation}
Given a non-decreasing weight $\prim= (s_m)_{m=1}^\infty$ we set
\[
H_m[\prim]=\sum_{n=1}^m \frac{s_n-s_{n-1}}{s_n}.
\]
If $\ww$ is a weight whose primitive weight $\prim$ is doubling, then for any $0<p<\infty$,
\begin{equation}
\label{eq:domLorenz}
\dom_m[d_{1,p}(\ww), d_{1,\infty}(\ww)]=(H_m[\prim])^{1/p}, \quad m\in\NN.
\end{equation}

\begin{remark}
A weight $\prim= (s_m)_{m=1}^\infty$ is said to have the \emph{upper regularity property} (URP for short) if there is $r\in\NN$ such that
\[
s_{rm}\le \frac{1}{2} r s_m, \quad m\in\NN,
\]
and is said to have the \emph{lower regularity property} (LRP for short) if there is $r\in\NN$ such that
\[
s_{rm}\ge 2 s_m, \quad m\in\NN.
\]
$\prim$ has the LRP if and only if $\prim^*=(m/s_m)_{m=1}^\infty$ has the URP. Moreover, if $\prim$ has the URP, there is a constant $C$ such that
\[
\sum_{n=1}^m \frac{1}{s_n} \le C \frac{m}{s_m}, \quad m\in\NN,
\]
(see \cite{DKKT2003}*{\S4}). Hence, if $\prim$ has the LRP, there is a constant $C$ such that
\[
\sum_{n=1}^m \frac{s_n}{n} \le C s_m= \sum_{n=1}^m s_n-s_{n-1}, \quad m\in\NN.
\]
Using that $1/s_{n}$ is non-increasing we infer that
\[
H_m:=\sum_{n=1}^m \frac{1}{n}= \sum_{n=1}^m \frac{s_n}{n}\frac{1}{s_n} \le C H_m[\prim],\quad m\in\NN.
\]
The reverse inequality holds for general doubling weights. Indeed, since $\inf_n s_n/s_{n+1}>0$, for every $\alpha>0$ there is a constant $C_1$ such that
\[
\frac{s_n-s_{n-1}}{s_n} \le C_1 \frac{s_n^\alpha -s_{n-1}^\alpha}{s_n^\alpha}, \quad n\in\NN.
\]
Moreover, for $\alpha$ small enough,
\[
C_2:=\sup_{n\le m} \frac{n}{m} \frac{s_m^\alpha}{s_n^\alpha}<\infty.
\]
Therefore
\[
s_m^\alpha \le C_2 \sum_{n=1}^m \frac{s_n^\alpha}{n},\quad m\in\NN.
\] Summing up,
\begin{equation*}
H_m[\prim] \le C_1 \sum_{n=1}^m \frac{s_n^\alpha -s_{n-1}^\alpha}{s_n^\alpha} \le C_1 C_2 H_m, \quad m\in\NN.
\end{equation*}
\end{remark}

\section{The Thresholding Greedy Algorithm, greedy parameters and duality}\label{Applications:Duality}\noindent 
An important research topic in approximation theory by means of bases is the study of the duality properties of the TGA. This section is motivated by the attempt to make headway in the following general question: If a basis $\XB$ enjoys some greedy-like property, what can be said about its dual basis $\XB^*$ in this regard? To that end, we need introduce the
\emph{bidemocracy parameters} of $\XB$,
\[
\bid_m[\XB,\XX]= \frac{1}{m} \usdf[\XB,\XX](m) \usdf[\XB^*,\XX^*](m), \quad m\in\NN.
\]
The basis $\XB$ is \emph{bidemocratic} (\cite{DKKT2003}) if and only if $\sup_m \bid_m<\infty$ (see e.g. \cite{AABW2019}*{Lemma 5.5}).

\begin{proposition}\label{prop:SSBiDem}
Let $\XB$ be a basis of a quasi-Banach space $\XX$. Then
\[
\max\{ \sq_m[\XB,\XX] , \sq_m[\XB^*,\XX^*]\} \le \bid_m[\XB,\XX], \quad m\in\NN.
\]
\end{proposition}

\begin{proof}
Let $f\in\XX$, $f^*\in\XX^*$, and $B\subseteq \NN$ with $|B|=m$. We have
\begin{align*}
\Vert f^* \Vert
&\ge\frac{f^*( \Ind_{\overline{\varepsilon(f^*)},B}[\XB,\XX])}{\usdf[\XB,\XX](m)}
=\frac{\sum_{n\in B} |f^*(\xx_n)|}{\usdf[\XB,\XX](m)}, \text{ and }\\
\Vert f \Vert
&= \frac{ \Vert f \Vert\, \Vert \Ind_{\overline{\varepsilon(f)},B}[\XB^*,\XX^*] \Vert}{\Vert \Ind_{\overline{\varepsilon(f)},B}[\XB^*,\XX^*] \Vert}
\ge\frac{\Ind_{\overline{\varepsilon(f)},B}[\XB^*,\XX^*](f)}{\usdf[\XB^*,\XX^*](m)}
=\frac{\sum_{n\in B} |\xx_n^*(f)|}{\usdf[\XB^*,\XX^*](m)}.
\end{align*}
Thus, if $|\{n\colon |x_{n}^{\ast}(f)|\ge 1\}|\ge m$ and $|\{n\colon |f^{\ast}(x_{n})|\ge 1\}|\ge m$ we infer that
\[
\Vert f\Vert \ge \frac{m}{\usdf[\XB^*,\XX^*](m)}, \quad \Vert f^*\Vert \ge \frac{m}{\usdf[\XB,\XX](m)}.
\]
We conclude the proof by applying Lemma~\ref{lem:SSParameters}~\ref{SSParameters:1}.
\end{proof}

Proposition~\ref{prop:SSBiDem} is a quantitative version of \cite{AABW2019}*{Proposition 5.7}. When combined with Theorem~\ref{thm:main} and Theorem~\ref{thm:MainNotMain} leads to linear estimates for the Lebesgue constants in terms of the bidemocracy parameters.

\begin{theorem}
Let $\XB$ be a basis of a quasi-Banach space $\XX$. Then, there are constants $C$ and $D$ such that for all $m\in\NN$,
\begin{equation}\label{eq:LebUncBid}
\leb_m \le C \max\{ \unc_m, \bid_m\},\end{equation}
and 
\[\leba_m \le D \max\{ \qg_m, \bid_m\}.\]
\end{theorem}
Inequality \eqref{eq:LebUncBid} was proved in the locally convex setting in \cite{AAB2020}*{Theorems 2.3 and 1.3} with the purpose of finding bounds for the growth of the greedy constant of the $L_p$-normalized Haar system as $p$ either increases $\infty$ or decreases to $1$.

Since the unconditionality parameters are defined in terms of linear operators, they dualize as expected, i.e.,
\[
\unc_m[\XB^*,\XX^*]\le \unc_m[\XB,\XX], \quad m\in\NN.
\]
The reverse inequality also holds in the case when $\XX$ is a Banach space. Consequently, by Theorem~\ref{thm:main}, for $m\in \NN$,
\begin{equation}\label{eq:GreedyDual}
\leb_m[\XB^*,\XX^*]\le C\max\{ \sq_m[\XB^*,\XX^*], \sq_m[\XB,\XX], \unc_m[\XB,\XX]\},
\end{equation}
where the constant $C$ depends only on the modulus of concavity of $\XX$. Our next goal is to obtain duality results for the almost greedy and quasi-greedy parameters.

\begin{proposition}\label{prop:QGDual}
Let $\XB$ be a basis of a quasi-Banach space $\XX$. Then for $m\in\NN$,
\begin{align*}
\qgc_m[\XB^*,\XX^*] &\le \qg_m[\XB,\XX] +\sq_m[\XB,\XX]+\sq_m[\XB^*,\XX^*]\text{ and }\\
\qg_m[\XB^*,\XX^*] &\le \qgc_m[\XB,\XX] +\sq_m[\XB,\XX]+\sq_m[\XB^*,\XX^*], \quad \end{align*}
\end{proposition}

\begin{proof}
Given $D\subseteq \NN$, put $S_D=S_D[\XB,\XX]$. Let $A$ be a greedy set of $f^*\in\XX^*$, and let $B$ be a greedy set of $f\in\XX$. Assume that $|A|=|B|\le m$. Then
\begin{align*}
|S_{B}^*(f^*)(S_{A^c}(f))|
&=\left|\sum_{n\in B\setminus A} f^*(\xx_n) \, \xx_n^*(f)\right|\\
&\le\min_{n\in A} | f^*(\xx_n)| \sum_{n\in B\setminus A} |\xx_n^*(f)|\\
&=\min_{n\in A} | f^*(\xx_n)| \, \Vert\Ind_{\varepsilon(f),B\setminus A}[\XB^*,\XX^*](f) \Vert\\
&\le \min_{n\in A} | f^*(\xx_n)| \, \Vert \Ind_{\varepsilon,B\setminus A}[\XB^*,\XX^*]\Vert \, \Vert f\Vert \\
&\le \sq_m[\XB^*,\XX^*]\, \Vert f^*\Vert \, \Vert f\Vert.
\end{align*}
Similarly, switching the roles of $\XB$ and $\XB^*$, we obtain
\[
|S_{A}^*(f^*)(S_{B^c}(f))|\le \sq_m[\XB,\XX] \Vert f^*\Vert \, \Vert f\Vert.
\]
Applying these inequalities to the identities
\begin{align*}\label{eq:DKKTIdea}
S_{A^c}^*(f^*)(f) &=f^*(S_{B^c}(f))+ S_{A^c}^*(f^*)(S_B(f))- S_{A}^*(f^*)(S_{B^c}(f))\\
S_{A}^*(f^*)(f)&= f^*(S_{B}(f))- S_{A^c}^*(f^*)(S_B(f)) + S_{A}^*(f^*)(S_{B^c}(f))
\end{align*}
leads to the desired inequalities.
\end{proof}

\begin{proposition}\label{cor:AGDual}
Let $\XB$ be a basis of a quasi-Banach space $\XX$. Threre are constant $C$, depending only on the modulus of concavity of $\XX$ such that
\[
\leba_m[\XB^*,\XX^*]\le C\max\{ \sq_m[\XB^*,\XX^*], \sq_m[\XB,\XX], \qg_m[\XB,\XX]\}, \quad m\in\NN.
\]
\end{proposition}

\begin{proof}
Just combine Theorem~\ref{thm:MainNotMain} with Proposition~\ref{prop:QGDual}.
\end{proof}

We make a stop \emph{en route} to gather some consequences of combining Theorem~\ref{thm:main}, Theorem~\ref{thm:MainNotMain}, Proposition~\ref{prop:SSBiDem}, Proposition~\ref{prop:QGDual}, Proposition~\ref{cor:AGDual} and inequality \eqref{eq:GreedyDual}.
\begin{theorem}
Let $\XB$ be a basis of a quasi-Banach space $\XX$. There are constants $C_1$, $C_2$, and $C_3$, depending only on the modulus of concavity of $\XX$, such that
\begin{align}
\qg_m[\XB^*,\XX^*] &\le C_1 \max\{\bid_m[\XB,\XX],\qg_m[\XB,\XX]\},\nonumber \\
\leba_m[\XB^*,\XX^*] &\le C_2 \max\{\bid_m[\XB,\XX],\leba_m[\XB,\XX]\}, \text{ and } \label{eq:AGDualBid}\\
\leb_m[\XB^*,\XX^*] &\le C_3 \max\{\bid_m[\XB,\XX],\leb_m[\XB,\XX]\}\nonumber.
\end{align}
\end{theorem}
Note that \eqref{eq:AGDualBid} is a quantitative version of \cite{DKKT2003}*{Theorem 5.4} (see also \cite{AABW2019}*{Corollary 6.8}).

Inequality~\eqref{eq:GreedyDual}, Proposition~\ref{prop:QGDual} and Proposition~\ref{cor:AGDual} justify the quest to find upper estimates for the squeeze symmetry parameters of the dual basis. We tackle this problem with the help of the bidemocracy paremeters.

Given a non-decreasing sequence $\prim= (s_m)_{m=1}^\infty$ we set
\[
R_m[\prim]=\frac{s_m}{m}\sum_{n=1}^m \frac{1}{s_n}, \quad m\in\NN.
\]
If $\prim$ has the URP, $\sup_m R_m[\prim]<\infty$. Thus the sequence $(R_m[\prim])_{m=1}^\infty$ could be said to measure the regularity of $\prim$. Note that, in the case when the dual weight $\prim^*$ is also non-decreasing,
\[
R_m[\prim]\le H_m, \quad m\in\NN.
\]
\begin{theorem}\label{eq:dualizeSS}
Let $0<p\le 1$, and let $\prim$ be the fundamental function of a basis $\XB=(\xx_n)_{n=1}^\infty$ of a $p$-Banach space $\XX$. Then, there are constants $C_1$ and $C_2$ depending only on $p$ such that
\begin{align*}
\bid_m[\XB,\XX] & \le C_1 \sq_m[\XB,\XX] R_m[\prim], \text{ and }\\
\unc_m[\XB,\XX] &\le C_2 \sq_m[\XB,\XX] (H_m[\prim])^{1/p}, \quad m\in\NN.
\end{align*}
\end{theorem}

\begin{proof}
Let $A\subseteq\NN$ with $|A|\le m$. Dualizing the operator $\Fou_A$ and taking into consideration \cite{CRS2007}*{Theorem 2.4.14}, we obtain that the operator
\[
T_A \colon d_{1,1}(1/\prim) \to \XX^*, \quad (a_n)_{n=1}^\infty \mapsto \sum_{n\in A} a_n\, \xx_n^*
\]
satisfies $ \Vert T_A\Vert \le \sq_m$. In particular,
\[
\Vert \Ind_{\varepsilon, A}[\XB^*,\XX^*] \Vert\le \sq_m \sum_{n=1}^m \frac{1}{s_n}, \quad \varepsilon\in\EE^A.
\]
This yields the estimate for the bidemocracy parameters. Now, by \cite{AABW2019}*{Theorem 9.12}, the unit vector system of $d_{1,p}(\ww)$ dominates $\XB$. Appealing to Lemma~\ref{lem:domunclocal} and the to the identity \eqref{eq:domLorenz} we obtain the estimate for the unconditionality parameters.
\end{proof}

To finish this section we will obtain estimates for the squeeze symmetric parameters in some particular situations that occur naturally in applications. Let us introduce a mild condition on bases.

\begin{definition}\label{def:UGH}
We say that a basis has the \emph{upper gliding hump property for constant coefficients} if there is a constant $C$ such that for every $A$ and $D\subseteq \NN$ finite there is $B\subseteq\NN$ with $A\cap D=\emptyset$, $|B|\le |A|$, and $\Vert \Ind_{A}\Vert \le C\Vert \Ind_{B}\Vert$.
\end{definition}

For instance, the trigonometric system in $L_1(\TT)$ or, more generally, in any translation invariant quasi-Banach space over $\TT$, has the upper gliding hump property for constant coefficients. Similarly any wavelet basis in a translation invariant quasi-Banach space over $\RR^d$ has the upper gliding hump property for constant coefficients.

\begin{lemma}\label{lem:UGH}
Suppose that a basis $\XB$ of a quasi-Banach space $\XX$ has the upper gliding hump property for constant coefficients. Then there is a constant $C$ such that
$\sq_m[\XB,\XX]\le C \sqd_m[\XB,\XX]$ for all $m\in\NN$.
\end{lemma}

\begin{proof}
Let $C$ be as in Definition~\ref{def:UGH}, and set $C_1=\Vert \Fou\Vert_{\XX\to c_0}$. Given $f\in\XX$ and $t>0$, let $B=\{n\in\NN \colon |\xx_n^*(f)|\ge t\}$, $s=\max\{ n\in\NN \setminus B \colon |\xx_n^*(f)|\}$. Suppose that $|B|\ge m$. Pick $A\subseteq\NN$ finite with $|A|\le |B|$ and $\epsilon>0$. There is $0<\epsilon_1 < (t-s)/(2C_1)$ such that $\Vert g\Vert \le \Vert f \Vert+\epsilon/(C\sqd_m)$ whenever $\Vert f -g\Vert\le \epsilon_1$. Use density to choose $g\in\XX$ with finite support. Then $B$ is a greedy set of $g$. Moreover, there is $D\subseteq\NN$ with $|D|\le |A|$, $D\cap\supp(g)=\emptyset$ and $\Vert \Ind_A\Vert \le C\Vert \Ind_D\Vert$. Therefore,
\[
t\Vert \Ind_A\Vert \le C \sqd_m \Vert g \Vert\le C \sqd_m \Vert f \Vert +\epsilon.
\]
Since $\epsilon$ is arbitrary, applying Lemma~\ref{lem:SignVsChar} puts an end to the proof.
\end{proof}

\begin{proposition}\label{prop:UGHPCC}
Let $\XB=(\xx_n)_{n=1}^\infty$ be a basis of a $p$-Banach space $\XX$, $0<p\le 1$. Suppose that $\XB$ has the upper gliding hump property for constant coefficients. Then, there is a constant $C$ depending only on $p$ such that $\leb_m \le C \leba_m (\log m)^{1/p}$ for all $m\ge 2$.
\end{proposition}

\begin{proof}
Just combine Lemma~\ref{lem:UGH}, Theorem~\ref{eq:dualizeSS}, Theorem~\ref{thm:main} and Theorem~\ref{thm:MainNotMain}.
\end{proof}

Given a basis $\XB$ of a quasi-Banach space $\XX$, the identity
\[
|A|=\Ind_{\varepsilon,A}[\XB^*,\XX^*] (\Ind_{\overline\varepsilon,A}[\XB,\XX]), \quad A\subseteq\NN,\; \varepsilon\in\EE^A,
\]
yields
\[
m \le \lsdf[\XB,\XX](m) \usdf[\XB^*,\XX^*](m), \quad m\in\NN.
\]
To quantify the optimality of this inequality, we introduce the parameters
\[
\wbid_m[\XB,\XX]= \frac{1}{m} \lsdf[\XB,\XX](m) \usdf[\XB^*,\XX^*](m), \quad m\in\NN.
\]
Since there are quite a few bases that satisfy the condition $\sup_m \wbid_m<\infty$, called property ($\bm{D}^*$) in \cite{BBGHO2018}, the following elementary lemma could be of interest.
\begin{lemma}
Let $\XB=(\xx_n)_{n=1}^\infty$ be a basis of a quasi-Banach space $\XX$. Then,
\[
\sq_m[\XB,\XX] \le \wbid_m[\XB,\XX]\, \sdem_m[\XB,\XX], \quad m \in\NN.
\]
\end{lemma}

\begin{proof}
Let $B$ be a greedy set of cardinality $m$ of $f\in\XX$. Set $t=\min_{n\in B} |\xx_n^*(f)|$. Let $A$ and $D$ be subsets of $\NN$ of cardinality $m$, and let $\varepsilon\in\EE^A$ and $\delta\in\EE^D$. Then,
\begin{align*}
t \Vert \Ind_{\varepsilon,A}[\XB,\XX] \Vert
&=\frac{tm\Vert \Ind_{\varepsilon,A}[\XB,\XX] \Vert \, \Vert \Ind_{\delta,D}[\XB,\XX] \Vert \, \Vert \Ind_{\overline{\varepsilon(f)},B}[\XB^*,\XX^*] \Vert }
{m\Vert \Ind_{\delta,D}[\XB,\XX] \Vert \, \Vert \Ind_{\overline{\varepsilon(f)},B}[\XB^*,\XX^*]\Vert}\\
&\le \sdem_m \frac{\Vert \Ind_{\delta,D} [\XB,\XX]\Vert \usdf[\XB^*,\XX^*](m) }{m}
\frac{ | \Ind_{\overline{\varepsilon(f)},B}[\XB^*,\XX^*](f) |}{\Vert \Ind_{\overline{\varepsilon(f)},B}[\XB^*,\XX^*]\Vert}\\
&\le \sdem_m \frac{\Vert \Ind_{\delta,D} [\XB,\XX] \Vert \usdf[\XB^*,\XX^*](m) }{m} \Vert f\Vert.
\end{align*}
Taking the infimum on $D$ and $\delta$ we obtain the desired inequality.
\end{proof}

Let us record an easy criterium for property ($\bm{D}^*$).
\begin{lemma}
Let $\XB$ a basis of a quasi-Banach space $\XX$ wich dominates a symmetric basis $\XB_1$ of a Banach space $\XX_1$. Suppose that there is a sequence $(A_m)_{m=1}^\infty$ in $\NN$ with $|A_m|=m$ for all $m\in\NN$, and
\[
\sup_m \frac{ \Vert \Ind_{A_m}[\XB,\XX]\Vert }{\usdf[\XB_1,\XX_1](m)}<\infty.
\]
Then $\XB$ has the property ($\bm{D}^*$).
\end{lemma}

\begin{proof}
Just dualize the operator from $\XX$ into $\XX_1$ provided by the domination hypothesis, and use that any symmetric basis of a locally convex space is bidemocratic (see \cite{LinTza1977}*{Proposition 3.a.6}).
\end{proof}

For instance, in the case when $\max\{p,q\}\ge 1$, the unit vector system of the mix-norm spaces $\ell_p\oplus \ell_q$, $(\bigoplus_{n=1}^\infty \ell_q^n)_{\ell_p}$ and $\ell_p(\ell_q)$ fulfils the above criterium. The trigonometric system in $L_p(\TT)$, $1<p\le \infty$, also does.

\begin{remark}
If a basis is democratic, its squeeze symmetry parameters and the truncation quasi-greedy parameters are of the same order. If a basis is either truncation quasi-greedy or has the property ($\bm{D}^*$), then its squeeze symmetry parameters and the superdemocracy parameters are of the same order. Hence, combining Theorem~\ref{thm:main}, Theorem~\ref{thm:MainNotMain}, Proposition~\ref{prop:SSBiDem}, Proposition~\ref{prop:QGDual}, Proposition~\ref{cor:AGDual}, \eqref{eq:GreedyDual} and Theorem~\ref{eq:dualizeSS} overrides \cite{BBGHO2018}*{Corollaries 7.2, 7.5 and 7.6}.
\end{remark}

\section{The spectrum of Lebesgue-type parameters associated to democratic bases}\label{sect:DLP}\noindent 
In greedy approximation theory, democracy can be regarded as an \emph{atomic} property of bases, in the sense that it cannot be broken into (or that it implies) other properties of interest in the theory. When combined with other properties of bases (atomic or otherwise) democracy gives rise to a spectrum of more complex types of bases whose hierarchy can be summarized in the following diagram:

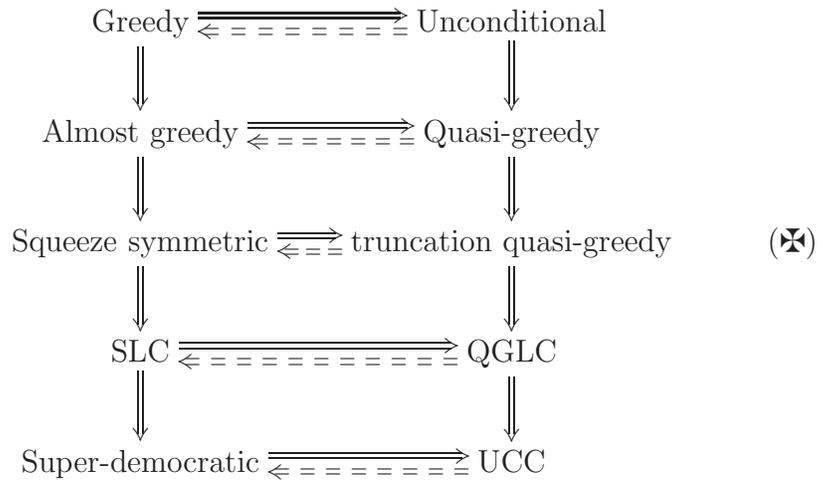
\begin{equation}\label{eq:Implications}\tag{$\maltese$}
\begin{gathered}
\xymatrix{
\text{Greedy} \ar@<3pt>@{=>}[r]\ar@{=>}[d]&\text{Unconditional} \ar@<3pt>@{==>}[l] \ar@{=>}[d]\\
\text{Almost greedy}\ar@<3pt>@{=>}[r]\ar@{=>}[d]&\text{Quasi-greedy}\ar@<3pt>@{==>}[l] \ar@{=>}[d]\\
\text{Squeeze symmetric}\ar@<3pt>@{=>}[r]\ar@{=>}[d]&\text{truncation quasi-greedy} \ar@<3pt>@{==>}[l] \ar@{=>}[d]\\
\text{SLC}\ar@<3pt>@{=>}[r]\ar@{=>}[d]&\text{QGLC} \ar@<3pt>@{==>}[l] \ar@{=>}[d]\\
\text{Super-democratic} \ar@<3pt>@{=>}[r]&\text{UCC} \ar@<3pt>@{==>}[l]\\
}
\end{gathered}
\end{equation}
When a property on the right hand-side column amalgamates with democracy it is transformed in the corresponding property on its left. For instance, democratic + unconditional $\Rightarrow$ greedy, democratic + quasi-greedy $\Rightarrow$ almost greedy, and so on.

Quantitatively, each implication in \eqref{eq:Implications} follows as a result of an estimate between the Lebesgue-type parameters associated to each property. Let us write down the relations between any two parameters associated to the properties from the left column of \eqref{eq:Implications}.

It is clear that
\begin{align}
\dem_m\le \sdem_m\le & \slc_m,\; \text{ and that } \label{eq:DLChain1}\\
&\slcd_m \le \leba_m \le \leb_m,\quad m\in\NN \label{eq:DLChain2}.
\end{align}
Inequalities \eqref{eq:DLChain1} and \eqref{eq:DLChain2} will be connected using the equivalence between the SLC parameters and the disjoint SLC parameters provided by the following proposition.

\begin{proposition}\label{prop:SLCDisjoint}
Let $\XB$ be a basis of a $p$-Banach space $\XX$, $0<p\le 1$. Then
\begin{equation} \label{eq:SLCDisjoint}
\slc_m[\XB,\XX]\le 2^{1/p-1} \FieldC^{1/p} A_p \slcd_m[\XB,\XX], \quad m\in\NN.
\end{equation}
\end{proposition}

\begin{proof}
Let $B\subseteq\NN$ with $|B|\le m$, $\delta\in\EE^B$, and $f\in\XX$ be finitely supported with $B\cap\supp(f)=\emptyset$. Pick an arbitrary extension of $\delta$ to $\EE^{\NN}$, which we still denote by $\delta$. Given $D\subseteq \NN$ with $|D|\le |B|$ and $D\cap \supp(f)=\emptyset$, we pick $E\subseteq \NN\setminus (B\cup D\cup \supp(f))$ with $|E|=|B|-|D|$. Set $g=\Ind_{\delta,D\cap B}+f$. Since the sets $D\setminus B$, $E$, $B\setminus D$ and $\supp(g)$ are pairwise disjoint, and $|D\setminus B|+|E|=|D\setminus B|$,
\begin{align*}
\Vert \Ind_{\delta,D}+ f\Vert^p
&\le 2^{-p} (\Vert \Ind_{\delta,D}+\Ind_{\delta,E}+ f\Vert^p+ \Vert \Ind_{\delta,D}-\Ind_{\delta,E}+ f\Vert^p)\\
&= 2^{-p} (\Vert \Ind_{\delta,D\setminus B}+\Ind_{\delta,E}+ g\Vert^p+ \Vert \Ind_{\delta,D\setminus B}-\Ind_{\delta,E}+ g\Vert^p)\\
&\le 2^{1-p} (\slcd_m)^p\Vert \Ind_{\delta,B\setminus D}+g\Vert^p\\
&= 2^{1-p}(\slcd_m)^p \Vert \Ind_{\delta,B}+f\Vert^p.
\end{align*}
Therefore, applying Lemma~\ref{lem:SignVsChar} gives the desired inequality.
\end{proof}


Next we compare the SLC parameters and the squeeze symmetry parameters. To that end we will use the relation between the QGLC and the truncation quasi-greedy parameters. Note that the parameters on the left column of \eqref{eq:Implications} (i.e., the unconditionality-like parameters) of a basis of a $p$-Banach space, $0<p\le 1$, satisfy
\begin{equation}\label{eq:UTChain}
\ucc_m\le \qglc_m \le \min\{\rto_m, \qg_m\}, \quad \max\{ A_p^{-1} \rto_m, \qg_m\}\le \unc_m.
\end{equation}
Inequalities \eqref{eq:RTOQG:1} and \eqref{eq:RTOQG:p} complete the quantitative estimates associated with the right column of \eqref{eq:Implications}.
\begin{proposition}
Let $\XB$ be a basis of a $p$-Banach space $\XX$, $0<p\le 1$. Then,
\begin{equation}\label{eq:SLCSQ}
\slc_m[\XB,\XX]\le 2^{1/p} \sq_m[\XB,\XX], \quad m\in\NN.
\end{equation}
\end{proposition}
\begin{proof}
Applying the $p$-triangle law gives $(\slc_m)^p \le (\sq_m)^p+(\qglc_m)^p$. Combining this inequality with \eqref{eq:UTChain} and \eqref{eq:SSDemRTO}, we are done.
\end{proof}

Combining \eqref{eq:DLChain1}, \eqref{eq:DLChain2}, \eqref{eq:SLCDisjoint} and \eqref{eq:SLCSQ} yields
\begin{equation}\label{eq:DLChain}
\dem_m\lesssim \sdem_m \lesssim \slc_m \lesssim\min\{\sq_m, \leba_m\} \le \max\{\sq_m, \leba_m\}\lesssim \leb_m.
\end{equation}

The only path for connecting with implications the squeeze symmetry parameters and the almost greediness parameters seems to be through the corresponding unconditionality-like properties. Indeed, combining \eqref{eq:SSDemRTO}, \eqref{eq:DLChain2}, \eqref{eq:RTOQG:1} and \eqref{eq:RTOQG:p} yields,
for every basis $\XB$ of a $p$-Banach space $\XX$,
\[
\sq_m[\XB,\XX] \le C (\leba_m[\XB,\XX])^{\alpha},\]
where
\[
\alpha=\begin{cases} 2 & \text{ if } p=1, \\ 2+1/p & \text{ if } p<1,\end{cases}\]
and the constant $C$ depends only on $p$. Thus, the question is whether this asymptotic estimate can be improved.
\begin{question}\label{qt:EA}
Given $0<p\le 1$, is there a constant $C$ such that $\sq_m\le C \leba_m$ for every basis of a $p$-Banach space?
\end{question}
Note that an (unlikely) positive answer to Question~\ref{qt:EA} would allow to replace $(\sqd_m)_{m=1}^\infty$ with $(\sq_m)_{m=1}^\infty$ in Theorem~\ref{thm:MainNotMain}. It would also provide an alternative proof to the estimate $\sq_m\lesssim \leb_m$ (see Theorem~\ref{thm:main}). In the same line of thought, we wonder about the relation between the squeeze symmetry parameters and their disjoint counterpart, as well as where to place the latter in inequality \eqref{eq:DLChain}.

\begin{question}
By Lemma~\ref{lem:SLCSSD}, $\sq_m \lesssim (\sqd_m)^2$. Hence, by \eqref{eq:DLChain}, $\slc_m\lesssim (\sqd_m)^2$, $\sdem_m\lesssim (\sqd_m)^2$ and $\dem_m\lesssim (\sqd_m)^2$. Can any of these asymptotic estimates be improved?
\end{question}

To finish this section we see the quantitative estimates associated with each row in \eqref{eq:Implications}. Inequalities \eqref{eq:UncG} and \eqref{eq:SSDemRTO} do the job for the first and the third rows, respectively. As far as the fifth row is concerned, it readily follows from \cite{AABW2019}*{Lemma 2.2} that
\begin{equation}\label{eq:DemSD}
\max\{ \dem_m, \ucc_m\} \lesssim \sdem_m \lesssim \dem_m \ucc_m,\quad m\in\NN.
\end{equation}

The following result takes care of the estimates involving the parameters in the fourth row.
\begin{proposition}
Let $\XB$ be a basis of a quasi-Banach space $\XX$. There are constants $C_1$ and $C_2$ depending only on the modulus of concavity of $\XX$ such that
\begin{equation*}
\frac{1}{C_2} \max\{ \dem_m, \qglc_m\} \le \slc_m \le C_1\dem_m \qglc_m, \quad m\in\NN.
\end{equation*}
\end{proposition}

\begin{proof}
Assume that $\XX$ is a $p$-Banach space. Let $f\in\XX$ with $\Vert\Fou(f)\Vert_\infty\le 1$, let $A\subseteq \NN$ with $A\cap \supp(f)=\emptyset$ and $|A|\le m$, and let $\varepsilon=(\varepsilon_n)_{n\in A}\in\EE^A$. We have
\[
\Vert \Ind_{\varepsilon,A} \Vert^p
\le 2^{-p}( \Vert \Ind_{\varepsilon,A}+f \Vert + \Vert \Ind_{\varepsilon,A}-f \Vert)
\le 2^{1-p}\slc_m \Vert \Ind_{\varepsilon,A}+f\Vert^p.
\]
This yields $\qglc_m\le 2^{1/p-1}\slc_m$. Thus, by \eqref{eq:DLChain1}, the proof of left-side inequality is over.

For every $D\subseteq A$,
\[
\Vert \Ind_{\varepsilon,D} \Vert \le \qglc_m \Vert \Ind_{\varepsilon,A}+ f \Vert.\] Therefore, by Lemma~\ref{lem:SignVsChar},
\[
\left\Vert\sum_{n\in A} a_n\, \xx_n\right\Vert \le \FieldC^{1/p} A_p \qglc_m \Vert \Ind_{\varepsilon,A}+ f \Vert, \quad |a_n|\le 1.
\]
Consequently, for any $E\subseteq\NN$ with $|E|\le |A|$,
\begin{align*}
\Vert \Ind_E +f \Vert^p
&\le \Vert \Ind_E\Vert^p + \Vert \Ind_{\varepsilon,A}\Vert^p + \Vert \Ind_{\varepsilon,A} + f \Vert^p \\
& \le \left( 1+ (\qglc_m)^p + \FieldC A_p^p (\dem_m)^p (\qglc_m)^p \right) \Vert \Ind_{\varepsilon,A}+ f \Vert^p.
\end{align*}
Applying again Lemma~\ref{lem:SignVsChar} puts an end to the proof.
\end{proof}
Finally, we tackle the quantitative estimates for the parameters in the second row of \eqref{eq:Implications}. Combining Theorem~\ref{thm:MainNotMain} with inequalities \eqref{eq:RTOQG:1}, \eqref{eq:RTOQG:p}, and \eqref{eq:SSDemRTO} gives
\[
\max\{ \dem_m, \qg_m\} \lesssim \leba_m \lesssim \dem_m (\qg_m)^\beta,\; \text{ where }\; \beta=\begin{cases} 1 & \text{ if } p=1, \\ 1+1/p & \text{ if } p<1.\end{cases}
\]
Notice that the relations between the Lebesgue parameters involved in the properties from the second row follow the same pattern as the relations of the parameters of the other rows in the diagram \eqref{eq:Implications} only in the locally convex setting.

\section{Examples}\label{examplesdeleixample}\noindent
Before we study the applicability of our estimates to important examples in Analysis, we need to introduce another type of democracy functions.

Let $\XB$ be a basis of a quasi-Banach space $\XX$. Lemma~\ref{lem:SignVsChar} implies that if we modify definition \eqref{eq:udf} by taking the supremum only over sets with $|A|=m$, we obtain a function equivalent to the upper democracy function; and the same occurs if we restrict ourselves to $\varepsilon=1$. In contrast, the function
\[
\udf[\XB,\XX](m)= \sup_{|A|=m} \Vert \Ind_A[\XB,\XX] \Vert, \quad m\in\NN.
\]
can be much smaller than $\usdf[\XB,\XX]$, whereas the non-decreasing function
\[
\ldf[\XB,\XX](m)=\sup_{1\le k \le m} \inf_{|A|=k} \Vert \Ind_A[\XB,\XX] \Vert, \quad m\in\NN
\]
can be much larger than $\lsdf[\XB,\XX]$ (see \cite{Woj2014}). Lemma~\ref{lem:SignVsChar} also gives the inequality
\begin{equation}\label{eq:uppervslower}
\usdf[\XB,\XX](m) \le \Upsilon^{1/p} \dem_m[\XB,\XX]\, \ldf[\XB,\XX](m), \quad m\in\NN
\end{equation}
for any basis $\XB$ of any $p$-Banach space $\XX$, $0<p\le 1$.

Given $p\in[1,\infty]$, we will denote by $p'$ its \emph{conjugate exponent}, determined by the identity
$1/p'=1-1/p$. We also set $p^*=\max\{p,p'\}$.

\subsection{Orthogonal systems as bases of $L_p$}\label{sect:OSLp}
Let $(\Omega,\Sigma,\mu)$ be a finite measure space and let $\XB=(\xx_n)_{n=1}^\infty$ be an orthogonal basis of $L_2(\mu)$. Given $1\le p\le\infty$ such that
$
\XB\subseteq L_{p^*}(\mu),
$
$\XB$ also is a basis of $L_p(\mu)$ (of its closed linear span if $r=\infty$).

\begin{lemma}\label{lem:upos}
Let $\XB$ be an orthonormal basis of $L_2(\mu)$, where $(\Omega,\Sigma,\mu)$ be a finite measure space. Let $1\le q<2< p\le \infty$ and suppose that the unit vector system of $\ell_q$ dominates $\XB$ regarded as a sequence in $L_p(\mu)$. Given $0\le \lambda \le 1$, we define $p_\lambda^+\in[2,p]$, $p_\lambda^-\in[p',2]$, $q_\lambda^+\in[q,2]$ and $q_\lambda^-\in[2,q']$ by
\[
\frac{1}{p_\lambda^+}=(1-\lambda)\frac{1}{p} +\frac{\lambda}{2},\quad
p_\lambda^-=(p_\lambda^+)', \quad
\frac{1}{q_\lambda^+}=(1-\lambda)\frac{1}{q} +\frac{\lambda}{2},\quad
q_\lambda^-=(q_\lambda^+)'.
\]
Then, there is a constant $C$ such that, for all $\lambda\in[0,1]$ and $\varepsilon=\pm 1$,
\[
\max\{\unc_m[\XB,L_{p_\lambda^\varepsilon}(\mu)], \sq_m[\XB, L_{p_\lambda^\varepsilon}(\mu)]\}
\le C m^{(1-\lambda)(1/q-1/2)}, \quad m\in\NN.
\]
\end{lemma}

\begin{proof}
We denote by $\XB_{(r)}$ the system $\XB$ regarded as a basic sequence in $L_r(\mu)$, $1\le r\le\infty$.
By Riesz-Thorin's interpolation theorem (see, e.g., \cite{Grafakos2004}), the unit vector system of $\ell_{q_\lambda^+}$ dominates $\XB_{(p_\lambda^+)}$. In turn, since $L_{p_\lambda^+}(\mu)$ is continuously included in $L_2(\mu)$, $\XB_{(p_\lambda^+)}$ dominates the unit vector system of $\ell_2$. By duality, the unit vector system of $\ell_2$ dominates $\XB_{(p_\lambda^-)}$, which, in turn, dominates the unit vector system of $\ell_{q_\lambda^-}$. Applying Lemma~\ref{lem:domunclocal}, and taking into account \eqref{eq:domlplq}, yields the desired result.
\end{proof}

For uniformly bounded orthogonal systems we obtain the following.
\begin{lemma}\label{lem:ububos}
Let $(\Omega,\Sigma,\mu)$ be a finite measure space and $\XB=(\xx_n)_{n=1}^\infty$ be an orthonormal basis of $L_2(\mu)$ with $\sup_n \Vert \xx_n\Vert_\infty<\infty$. There is a constant $C$ such that, for all $1\le p \le \infty$,
\[
\frac{1}{C}\max\{\unc_m[\XB,L_p(\mu)], \sq_m[\XB, L_p(\mu)]\} \le \Phi_p(m):=m^{|1/p-1/2|}, \quad m\in\NN.
\]
\end{lemma}

\begin{proof}
Just apply Lemma~\ref{lem:upos}, taking into account that the unit vector system of $\ell_1$ dominates $\XB_{(\infty)}$.
\end{proof}

Let us obtain lower estimates for the parameters.
\begin{lemma}\label{lem:lbos}
Let $\XB$ be an orthogonal basis of $L_2(\mu)$, where $(\Omega,\Sigma,\mu)$ is a finite measure space. Let $1\le p\le\infty$ be such that $\XB\subseteq L_{p^*}(\mu)$.
\begin{enumerate}[label=(\roman*), leftmargin=*, widest=iii]
\item\label{los:1} If $1\le p \le 2$, there is a constant $C$ such that
\begin{align*} 
\ucc_m[\XB,L_{p}(\mu)]&\ge \frac{1}{C} \frac{m^{1/2}}{\lsdf[\XB,L_{p}(\mu)](m)} \text{ and }\\
\dem_m[\XB,L_{p}(\mu)]&\ge \frac{1}{C} \frac{m^{1/2}}{\ldf[\XB,L_{p}(\mu)](m)}, \quad m\in\NN.
\end{align*}
Moreover, $\XB$ has the upper gliding hump property for constant coefficients.
\item\label{los:2} If $2\le p <\infty$, there is a constant $C$ such that
\begin{align*}
\ucc_m[\XB,L_p(\mu)] & \ge \frac{1}{C} \frac{ \usdf[\XB,L_p(\mu)](m) }{ m^{1/2}} \text{ and}\\
\dem_m[\XB,L_p(\mu)]&\ge \frac{1}{C} \frac{ \udf[\XB,L_p(\mu)](m) }{ m^{1/2}}, \quad m\in\NN.
\end{align*}
\item\label{los:3} There is a constant $C$ such that
\[
\ucc_m[\XB,L_\infty(\mu)]\ge \frac{1}{C} \frac{m}{\lsdf[\XB,L_\infty(\mu)](m)}, \quad m\in\NN.
\]
\end{enumerate}
\end{lemma}

\begin{proof}
If $1\le p< \infty$, by Kahane-Khintchine inequalities and \cite{AACV2019}*{Proposition 2.4}, there is a constant $C_1$ such that
\[
\frac{1}{C_1} |A|^{1/2} \le \Ave_{\varepsilon\in\EE^A} \Vert \Ind_{\varepsilon,A} \Vert_{p} \le C_1 |A|^{1/2},\quad A\subseteq\NN.
\]
Combining these inequalities with Lemma~\ref{lem:SignVsChar} yields \ref{los:1} and the estimate for $\ucc_m$ in \ref{los:2}. In the case when $2< p<\infty$, by \cite{KadecPel1962}*{Corollary 7}, $\XB$ has a subsequence equivalent to the unit vector system of $\ell_2$. We infer that the estimate for $\dem_m$ in \ref{los:2} holds. If $p=\infty$, since $c=\inf_n \Vert \xx_n\Vert_1>0$ (see, e.g., \cite{AACV2019}*{Lemma 2.7}),
\[
\int_\Omega \left(\sum_{n\in A} |x_n|\right)\, d\mu \ge c |A|,
\]
whence
$
\sup_{\varepsilon\in\EE^A} \Vert \Ind_{\varepsilon,A}\Vert_\infty\ge c |A|,
$
for all $A\subseteq \NN$ finite. We infer that \ref{los:3} holds.
\end{proof}

\subsection{The trigonometric system over $\TT^d$} Temlyakov \cite{Temlyakov1998b}*{Theorem 2.1 and Remark 2} established the growth of the Lebesgue constants of the trigonometric system in $L_p$. Later on, Wojtaszczyk \cite{Woj2000}*{Corollary~(a)} and Blasco et al.\ \cite{BBGHO2018}*{Proposition 8.6} revisited this result. Our discussion uses Theorem~\ref{thm:main} and, more specifically, the estimates obtained in Sections~\ref{sect:OSLp}. Note that Lemmas~\ref{lem:ububos} and \ref{lem:lbos} apply, in particular, to the trigonometric system $\Tt^d$ over $\TT^d:=\RR^d/\ZZ^d$.

Using Shapiro's polynomials we obtain for each $m\in\NN$, a set $A_m$ of cardinality $m$ with $\Vert \Ind_{\varepsilon,A_m} \Vert_\infty \le 5 m^{1/2}$ for a suitable and $\varepsilon\in\EE^{A_m}$ (see \cite{Rudin1959}). Hence, there is a constant $C_1$ such that
\[
\ucc_m[\Tt^d,L_\infty(\TT^d)] \ge \frac{1}{C_1} |m|^{1/2}, \quad m\in\NN.
\]
In the case when $1<p<\infty$, Dirichlet's kernel (see e.g.\ \cite{Grafakos2004}) shows the existence of a constant $C_2$ such that, for the same sets $A_m$,
\[
m^{1-1/p}/C_2\le \Vert \Ind_{A_m}\Vert_p\le C_2 m^{1-1/p}.\]
Therefore, for each $1<p< \infty$ there is a constant $C_3$ such that
\[
\min\{ \dem_m[\Tt^d,L_p(\TT^d)] , \ucc_m[\Tt^d,L_p(\TT^d)]\} \ge \frac{1}{C_3} \Phi_p(m), \quad m\in\NN.
\]

In the case when $p=1$, we have
\[
\Vert \Ind_{A_m}\Vert_1\le C_4 \log m\] for all $m\ge 2$ and a suitable constant $C_4$. Thus, the same argument gives
\[
C_5 \min\{ \dem_m[\Tt^d,L_1(\TT^d)], \ucc_m[\Tt^d,L_1(\TT^d)]\} \ge \Psi(m):= \frac{m^{1/2}}{\log m}, \quad m\ge 2,
\]
for a suitable constant $C_5$. This estimate is optimal. Indeed, applying induction on $d$ and using Fubini's theorem, we infer from \cite{MPS1981}*{Theorem 2.1} that there is a constant $C_6$ such that
\[
\Vert \Ind_{\varepsilon,A} [\Tt^d,L_1(\TT^d)]\Vert_1 \ge\frac {1}{C_6} \log(|A|), \quad A\subseteq\NN,\; \varepsilon\in\EE^A.
\]
Therefore, using the aforementioned bounded linear map from $\ell_2$ into $L_1(\TT^d)$,
\[
\sdem_m[\Tt^d,L_1(\TT^d)]\le C_6 \Psi(m), \quad m\ge 2.
\]
To obtain sharp estimates for the squeeze symmetry parameters and the unconditionality parameters in the case $p=1$, we invoke the De Vall\'ee-Pousin's kernel, which yields, for each $m\in\NN$ and $s>1$, a function $v_{m,s}$ with $\Vert v_{m,s} \Vert_1 \le 1+s$ and
\[
\chi_{A_m} \le \Fou(v_{m,s}) \le 1,
\]
(see, e.g., \cite{Mehta2015}). Since there is a constant $C_7$ such that for each $m\in\NN$ there is $\varepsilon\in\EE^{A_m}$ with $m^{1/2} \le C_7 \Vert \Ind_{\varepsilon, A_m}\Vert_1$, applying Lemma~\ref{lem:SignVsChar} we obtain
\[
C_7\min\{ \sqd_m [\Tt^d,L_1(\TT^d)] , \FieldC \qglc_m [\Tt^d,L_1(\TT^d)] \} \ge m^{1/2}, \quad m\in\NN.
\]
Summing up, the democracy, super-democracy, SLC, disjoint squeeze symmetry, squeeze symmetry, almost greediness, Lebesgue, unconditionality, quasi-greediness, truncation quasi-greedy, QGLC and UCC parameters of the trigonometric system in $L_p(\TT^d)$ grow as $(\Phi_p(m))_{m=1}^\infty$ for all $1\le p \le \infty$ with the follwing exceptions in the case $p=1$: the democracy, super-democracy and the UCC parameters of the trigonometric system in $L_1(\TT^d)$ grow as $(\Psi(m))_{m=2}^\infty$; and in the case $p=\infty$, $\Tt^d$ is democratic in $L_\infty(\TT^d)$.

\subsection{The trigonemetric system in Hardy spaces} Fix $0<p<1$.
If for $n\in\NN\cup\{0\}$ we set
\[
\tau_n(\theta)=e^{2\pi i \theta}, \quad -1/2\le\theta\le 1/2,
\]
the sequence $\Tt=(\tau_n)_{n=0}^\infty$ is a basis of $H_p(\TT)$ whose biorthogonal functionals are the members of the sequence $(\overline{\tau}_n)_{n=0}^\infty$ under the natural dual mapping. Since $H_2(\TT)\subseteq H_p(\TT)$, the unit vector system of $\ell_2$ dominates $\Tt$ regarded as basis of $H_p(\TT)$. In turn, since the dual basis is uniformly bounded (see \cite{DRS1969}), $\Tt$ dominates the unit vector system of $c_0$. We infer from Lemma~\ref{lem:domunclocal} that there is a constant $C$ such that
\[
\max\{\unc_m[\Tt,H_p(\TT)], \sq_m[\Tt,H_p(\TT)]\} \le C m^{1/2}, \quad m\in\NN.
\]
This estimates are optimal. Indeed, the Dirichlet kernel $\sum_{k=0}^{n-1}\tau_k$, $n\in\NN$, is uniformly bounded in $H_p(\TT)$, and Khintchine's inequalities yield a constant $C_1$ such that
\[
\frac{1}{C_1} |A|^{p/2} \le \Ave_{\varepsilon\in\EE^A} \Vert \Ind_{\varepsilon,A} \Vert^p_{H_p} \le C_1 |A|^{p/2},\quad A\subseteq\NN.
\]
Therefore, for some constant $C_2$,
\[
\min\{\dem_m[\Tt,H_p(\TT)], \ucc_m[\Tt,H_p(\TT)]\} \ge\frac{1}{C_2} m^{1/2}, \quad m\in\NN.
\]

\subsection{Jacobi polynomials}
Given scalars $\alpha$, $\beta>-1$, the Jacobi polynomials
\[
\JB(\alpha,\beta)=(p_n^{(\alpha,\beta)})_{n=0}^\infty
\]
appear as the orthonormal polynomials associated with the measure $\mu_{\alpha,\beta}$ given by
\begin{equation*}
d\mu_{\alpha,\beta}(x)=(1-x)^\alpha (1+x)^\beta \chi_{(-1,1)} (x) \, dx.
\end{equation*}
In the case when $\gamma_0:=\min\{\alpha,\beta\}> -1/2$ we set $\gamma=\max\{\alpha,\beta\}$ and
\[
\underline{p}=\underline{p}(\alpha,\beta)=\frac{4(\gamma+1)}{2\gamma+3}, \quad
\overline{p}=\overline{p}(\alpha,\beta)=\frac{4(\gamma+1)}{2 \gamma +1}.
\]
Notice that $\underline{p}$ and $\overline{p}$ are conjugate exponents. Given $p\in (\underline{p},\overline{p})$ we define $q(p,\alpha,\beta)\in(1,\infty)$ by
\[
\frac{1}{q(p,\alpha,\beta)}=\lambda,
\]
where $\lambda\in[0,1]$ is such that
\[
p=(1-\lambda)\underline{p}(\alpha,\beta)+\lambda \overline{p}(\alpha,\beta).
\]
A routine computation yields
\[
\frac{1}{q(p,\alpha,\beta)}=\frac{2\gamma+3}{2}-\frac{2(\gamma+1)}{p}, \quad
\underline{p}(\alpha,\beta) <p< \overline{p}(\alpha,\beta).
\]
We also define $r=r(p,\alpha,\beta)$ by
\[
\frac{1}{r(p,\alpha,\beta)}=\frac{2\gamma_0+3}{2}-\frac{2(\gamma_0+1)}{p}, \quad
\underline{p}(\alpha,\beta) <p< \overline{p}(\alpha,\beta).
\]

\begin{theorem}\label{UPJacobi}
Let $\alpha$ and $\beta$ be such that $\min\{\alpha,\beta\}> -1/2$. Given $p\in (\underline{p}(\alpha,\beta),\overline{p}(\alpha,\beta))$, set $q=q(p,\alpha,\beta)$. In the case when $p\le 2$, the unit vector system of $\ell_q$ dominates $\JB(\alpha,\beta)$ regarded as sequence in $L_p(\mu_{\alpha,\beta})$ and, in the case when $p\ge 2$, $\JB(\alpha,\beta)$, regarded as sequence in $L_p(\mu_{\alpha,\beta})$, dominates the unit vector system of $\ell_q$.
\end{theorem}

\begin{proof}
This result could be derived from \cite{Veprintsev2016}. Here, we present an alternative proof. Using Marcinkiewicz's interpolation theorem (see, e.g., \cite{Grafakos2004}) and duality, it suffices to prove that the unit vector system of $\ell_1$ dominates $(p_n^{(\alpha,\beta)})_{n=0}^\infty$ regarded as a sequence in $\XX:=L_{\overline{p},\infty} (\mu_{\alpha,\beta})$. Since $\XX$ is locally convex, we must prove that $\sup_n \Vert p_n^{(\alpha,\beta)}\Vert_{\overline{p},\infty} <\infty$. This can be deduced from classical estimates for Jacobi polynomials (see \cite{AACV2019}*{Theorem 3.2 and Lemma 3.3}) or from the fact that the partial sums of Jacobi-Fourier series $(J_n)_{n=1}^\infty$ are uniformly bounded when regarded as operators from $L_{\overline{p},1}(\mu_{\alpha,\beta})$ into $L_{\overline{p},\infty}(\mu_{\alpha,\beta})$ (see \cite{GPV1990}). Indeed, taking into account that the dual of $L_{\overline{p},1}(\mu_{\alpha,\beta})$ is $L_{\underline{p},\infty}(\mu_{\alpha,\beta})$ under the natural dual pairing, the uniform boundedness of the operators $(J_n-J_{n-1})_{n=0}^\infty$ yields
\[
\sup_{n\in\NN} \Vert p_n^{(\alpha,\beta)}\Vert_{\underline{p},\infty}
\Vert p_n^{(\alpha,\beta)}\Vert_{\overline{p},\infty}<\infty.
\]
Since $L_{\overline{p},\infty}(\mu_{\alpha,\beta})\subseteq L_1(\mu_{\alpha,\beta})$ and $\inf_n \Vert p_n^{(\alpha,\beta)}\Vert_1>0$ (see, e.g., \cite{AACV2019}*{Equation (3.4)}), we are done.
\end{proof}

\begin{theorem}\label{thm:LBJacobi}
Let $\alpha$ and $\beta$ be such that $\min\{\alpha,\beta\}> -1/2$. Given $p\in [2,\overline{p}(\alpha,\beta))$, set $q=q(p,\alpha,\beta)$ and $r=r(p,\alpha,\beta)$. Then:
\begin{enumerate}[label=(\roman*), leftmargin=*, widest=ii]
\item If $p\ge 2$, there is a constant $C_1$ such that
\[
\udf[\JB(\alpha,\beta),L_p(\mu_{\alpha,\beta})](m)\ge \frac{1}{C_1} m^{1/q}, \quad m\in\NN,
\]
\item If $p\le 2$, there is a constant $C_2$ such that
\[
\lsdf[\JB(\alpha,\beta),L_p(\mu_{\alpha,\beta})](m)\le C_2 m^{1/r}, \quad m\in\NN.
\]
\end{enumerate}
\end{theorem}

\begin{proof}
It follows by combining \cite{AACV2019}*{Proposition 3.8}, Lemma~\ref{lem:SignVsChar}, and the fact that $\JB(\alpha,\beta)$ is a Schauder basis of $L_p(\mu_{\alpha,\beta})$ (see \cite{Pollard1949}).
\end{proof}

\begin{theorem}
Let $\alpha$ and $\beta$ be such that $\min\{\alpha,\beta\}> -1/2$. Set $\gamma_0=\min\{\alpha,\beta\}$ and $\gamma=\max\{\alpha,\beta\}$.
\begin{enumerate}[label=(\roman*), leftmargin=*, widest=ii]
\item If $p\in [2,\overline{p}(\alpha,\beta))$, the democracy, super-democracy, SLC, disjoint squeeze symmetry, squeeze symmetry, almost greediness, Lebesgue, unconditionality, quasi-greediness, truncation quasi-greedy, QGLC and UCC parameters of $\JB(\alpha,\beta)$ regarded as a basis of $L_p(\mu_{\alpha,\beta})$ grow as
\[
\Phi(m)=m^{(1+\gamma)|1-2/p|}, \quad m\in\NN.
\]
\item If $p\in (\underline{p},2]$, the Lebesgue constants and the unconditionality parameters of $\JB(\alpha,\beta)$ regarded as a basis of $L_p(\mu_{\alpha,\beta})$ grow as the sequence $(\Phi(m))_{m=1}^\infty$;
the almost greedy parameters, the squeeze symmetry, and the disjoint squeeze symmetry parameters grow at least as
\[
\frac{\Phi(m)}{\log m}, \quad m\ge 2;
\]
and the super-democracy, SLC, quasi-greediness, truncation quasi-greedy, QGLC and UCC parameters grow as least as
\[
m^{(1+\gamma_0)|1-2/p|}, \quad m\in\NN.
\]
\end{enumerate}
\end{theorem}

\begin{proof}
Just combine Theorems~\ref{thm:LBJacobi} and \ref{UPJacobi}, Lemmas~\ref{lem:upos} and \ref{lem:lbos}, and Proposition~\ref{prop:UGHPCC}.
\end{proof}

\subsection{Lindenstrauss dual bases}\label{sect:LDB}
Let $\delta=(d_n)_{n=1}^\infty$ be a non-decreasing sequence in $\NN$ with $d_n\geq 2$ for all $n\in\NN$. Set
\begin{equation*}
\quad \sigma(k)=2+\sum_{j=1}^{k-1} d_j, \quad k\in\NN.
\end{equation*}
Let $\Gamma\colon \NN\to\NN\cup\{0\}$ be the left inverse of the function defined by $n\mapsto\sigma^{(n)}(1)$, $n\in\NN\cup\{0\}$. In \cite{AAW2020} it was constructed an almost greedy basis $\XB_\delta$ of a subspace $\XX_\delta$ of $\ell_1$ with
\[
\frac{1}{4}(1+\Gamma(m)) \le \unc_m[\XB_\delta,\XX_\delta] \le 2 (1+\Gamma(m)), \quad m\in\NN.
\]
In the case when $d_n=2$ for $n\in\NN$ the resulting space is the classical Lindesntrauss space, say $\XX$, built in \cite{Lin1964}. Moreover $\XX_\delta$ is isomporphic to $\XX$ regardless the choice of $\delta$. The dual space of $\XX_\delta$ is isomorphic to $\ell_\infty$, and the dual basis $\XB_\delta^*$ spans a space isomorphic to $c_0$. In \cite{AAW2020} it is also proved that for each increasing concave function $\phi\colon[0,\infty) \to [0,\infty)$ with $\phi(0)=0$, we can choose $\delta$ so that $\Gamma$ grows as $(\phi(\log (m)))_{m=2}^\infty$. By \cite{AAW2020}*{Proposition 4.4 and Lemma 7.3},
\[
\bid_m[\XB_\delta, \XX_\delta]\le 2 (1+\Gamma(m)), \quad m\in\NN.
\]
Thus, by \eqref{eq:GreedyDual} and Proposition~\ref{prop:SSBiDem}, there is a constant $C$ such that $\leb_m[\XB_\delta^*, \XX_\delta^*]\le C\Gamma(m)$ for all $m\ge 2$.

As far as lower bounds is concerned, combining \cite{AAW2020}*{Lemma 7.1 and Lemma 7.2} yields
\[
\dem_m[\XB_\delta^*, \XX_\delta^*]\ge \frac{1}{2} \Gamma(m), \quad m\in\NN;
\]
and the proof of \cite{AAW2020}*{Lemma 7.1} gives that
\[
\ucc_m[\XB_\delta^*, \XX_\delta^*]\ge\frac{1}{8} \Gamma(m), \quad m\in\NN.
\]

\subsection{Bases with large greedy-like parameters}
Let $\XB=(\xx_n)_{n=1}^\infty$ be a basis of a $p$-Banach space satisfying \eqref{eq:finiteparameters}. Then,
\[
\max\{ \unc_m[\XB,\XX], \sq_m[\XB,\XX]\} \le (C[\XB])^2 m^{1/p}, \quad m\in\NN.
\]
Consequently, there is a constant $C$ such that $\leb_m[\XB,\XX]\le C m^{1/p}$ for all $m\in\NN$. There are bases for which this estimate is optimal. Take, for instance, the difference basis $\DB=(\dd_n)_{n=1}^\infty$ in $\ell_p$ given by
\[
\dd_n=\ee_n-\ee_{n-1},\quad n\in \NN,
\]
where $(\ee_n)_{n=1}^\infty$ is the unit vector system and $\ee_0=0$.
For $0<p\le 1$, $\DB$ is a Schauder basis of $\ell_p$ whose dual basis is (naturally identified with) the summing basis $\SB=(\sss_n)_{n=1}^\infty$ of $c_0$ given by
\[
\sss_n=\sum_{k=1}^n \ee_k,\quad n\in \NN.
\]
Since $\Vert \sum_{n=1}^m \dd_n\Vert_p=1$ and $\Vert \sum_{n=1}^{m} \dd_{2n} \Vert_p=(2m)^{1/p}$ for all $m\in\NN$, we have
\[
\dem_m[\DB,\ell_p]\ge (2m)^{1/p},\quad m\in \NN,
\]
and
\[\ucc_{m}[\DB,\ell_p]\ge m^{1/p}, \quad m\in \NN.
\]

As for the dual basis, we have 
\[
\left\Vert \sum_{n=1}^m \sss_n\right\Vert_\infty=m,\;\text{and}\; \left\Vert \sum_{n=1}^m (-1)^n\sss_n\right\Vert_\infty=1,
\]
whence $\ucc_m[\SB,c_0]\ge m$, for all $m\in\NN$. Notice that the summing basis of $c_0$ is democratic.

Another classical basis with large greedy-like parameters is the $L_1$-normalized Haar system $\HB$. It is essentially known that
\[
\dem_{m}[\HB,L_1([0,1])]\ge m, \quad m\in\NN,
\]
and
\[\ucc_{2m}[\HB,L_1([0,1])\ge m/4,\quad m\in\NN,
\]
(see \cite{DKW2002}). Nonetheless, certain subbases of $\HB$ are quasi-greedy basic sequences in $L_1([0,1])$ \cites{DKW2002,Gogyan2016}.

\subsection{Tsirelson's space}
The space $\Ts^*$ constructed by Tsirelson \cite{Tsirelson1974} to prove the existence of a Banach space that contains no copy of $\ell_p$ or $c_0$ is the dual of the Tsirelson space $\Ts$ defined by Figiel and Johnson \cite{FigielJohnson1974}. The unit vector system $\BB=(\ee_n)_{n=1}^\infty$ is a greedy basis of $\Ts$ whose fundamental function is equivalent to the fundamental function of the unit vector system of $\ell_1$ (see \cite{DOSZ2011}). Although $\Ts$ contains no copy of $\ell_1$, its unit vector system contains finite subbases uniformly equivalent to the unit vector system of $\ell_1^n$ for all $n\in\NN$ (see \cite{CasShu1989}*{Proposition I.2}). Thus, the unit vector system of the original Tsirelson space $\Ts^*$ contains finite subbases uniformly equivalent to the unit vector system of $\ell_\infty^n$ for all $n\in\NN$. In particular, $\ldf[\BB,\Ts^*]$ is bounded. With an eye to studying the TGA with respect to the canonical basis of the original Tsirelson space $\Ts^*$ we give a general lemma.

\begin{lemma}\label{lem:closel1}
Let $\XB$ be a basis of a quasi-Banach space $\XX$. Suppose that $\ldf[\XB^*,\XX^*]$ is bounded and that $\XB^{**}$ is equivalent to $\XB$. Then
\begin{align*}
\bid_m[\XB,\XX]
&\approx\leba_m[\XB^*,\XX^*]
\approx \sq_m[\XB^*,\XX^*]
\approx \sqd_m[\XB^*,\XX^*]
\approx\usdf[\XB^*,\XX^*]\\
&\approx \dem_m[\XB^*,\XX^*]
\approx\dom_m[\XB^*,c_0]
\approx \dom_m[\ell_1,\XB].
\end{align*}
\end{lemma}
\begin{proof}
The equivalence $\dom_m[\XB^*,c_0]\approx \dom_m[\ell_1,\XB]$ follows by duality. The inequality
\[
\bid_m[\XB,\XX]\le \usdf[\XB^*,\XX^*](m) \sup_n\Vert \xx_n\Vert
\]
holds for any basis of any Banach space. Combining Lemma~\ref{lem:SSParameters}, Propositions~\ref{prop:SSDemRTO} and \ref{prop:SSBiDem}, and inequalities \eqref{eq:usdfdoml00} and \eqref{eq:uppervslower} concludes the proof.
\end{proof}

Loosely speaking, Lemma~\ref{lem:closel1} says that, for bases close to canonical $\ell_1$-basis, the squeeze symmetry parameters of their dual basis measure how far the basis is from the unit vector system of $\ell_1$. We note that this applies in particular to the Lindenstrauss bases we considered in \S\ref{sect:LDB}.

By Theorem~\ref{eq:dualizeSS}, the dual basis $\XB^*$ of any squeeze symmetric basis $\XB$ of as quasi-Banach space $\XX$ satisfies $\leb_m[\XB^*,\XX^*]=O(\log m)$. For the canonical basis of the original Tsirelson space (which is not greedy) this general estimate is far from being optimal. To write down a precise statement of this estimate, we recursively define $\log^{(k)}\colon (e^{k-1},\infty) \to (0,\infty)$ by $\log^{(1)}=\log$ and
\[
\log^{(k)}=\log^{(k-1)} \circ \log.
\]
Since $\BB$ is an unconditional basis of $\Ts^*$, applying Lemma~\ref{lem:closel1} to the canonical basis of $\Ts$ yields $\leb_m[\BB,\Ts^*]\approx \dom_m[\ell_1,\Ts]$. Moreover, by \cite{CasShu1989}*{Proposition I.9.3},
\[
\dom_m[\ell_1,\Ts]\approx\sup\left\{\sum_{n=1}^m |a_n| \colon \Vert \sum_{n=1}^m a_n \ee_n \Vert_{\Ts} \le 1\right\}, \quad m\in\NN.
\]
Then, by \cite{CasShu1989}*{Proposition IV.b.3},
\[
\lim_m \frac{\leb_m[\BB,\Ts^*]}{\log^{(k)}(m)}=0
\]
for all $k\in\NN$.

\subsection{The dual basis of the Haar system in $\BV(\RR^d)$}\label{example:BV}
Given $d\in\NN$, $d\ge 2$, let $\Dy$ denote the set consisting of all dyadic cubes in the Euclidean space $\RR^d$. If $Q\in\Dy$ there is a unique $k=k(Q)\in\ZZ$ such that $|Q|=2^{-kd}$.
Given $P\in\Dy$ and $k\in\ZZ$ we define
\[
\Qy[P,k]=\{Q\in \Dy \colon Q\subseteq P,\; k(Q)=k\}.
\]
Of course, $\Qy[P,k]=\emptyset$ for all $k<k(P)$. Set also
\[
\Dy[P,k]=\bigcup_{j=k(P)}^{k-1} \Qy[P,j], \quad k>k(P).
\]

Given an interval $J\subseteq\RR$ we denote by $J_{l}$ its left-half and by $J_{r}$ its right-half, and we set $h_I^0=\chi_I$ and $h_J^1=-\chi_{J_{l}}+\chi_{J_{r}}$. For $\theta=(\theta_i)_{i=1}^d\in \Theta_d:=\{0,1\}^d\setminus\{0\}$ and $Q\in \prod_{i=1}^d J_i\in\Dy$ put
\[
h_{Q,\theta}=|Q|^{(1-d)/d}\prod_{j=1}^d h_{J_i}^{\theta_i}, \quad h_{Q,\theta}^*=|Q|^{-1/d}\prod_{i=1}^d h_{J_i}^{\theta_i},
\]
and we denote by $\XX$ be the subspace of $\BV(\RR^d)$ spanned by
\[
\HB=(h_{Q,\theta})_{(Q,\theta)\in\Dy\times\Theta_d}.
\]
Let $\Ave(f;Q)$ denote average value of $f\in L_1(\RR^d)$ over the cube $Q$. For every $f\in \BV(\RR^d)$, $P\in\Dy$, and $k\in\ZZ$, $k> k(P)$, we have
\begin{align*}
T_{P,k}(f)&:= \sum_{Q\in\Dy[P,k]}\sum_{\theta\in\Theta_d} h_{Q,\theta} \int_{\RR^d} f(x)\, h_{Q,\theta}^*(x)\, dx\\
&=-\Ave(f;P)\chi_P+ \sum_{Q\in\Qy[P,k]} \Ave(f;Q) \chi_Q.
\end{align*}
Hence, if $\pi\colon\NN\to\Dy\times \Theta_d$ is a bijection such that the sets
\[
\pi^{-1}(\{P\}\times \Theta_d)\quad\text{and}\quad
\pi^{-1}(\Qy[P,k(P)+1]\times \Theta_d)
\]
are integer intervals for every $P\in\Dy$, applying \cite{Woj2003}*{Corollary 12} gives that
$
(h_{\pi(n)})_{n=1}^\infty
$
is a semi-normalized Schauder basis of $\XX$. By \cite{CDVPX1999}*{Theorem 8.1 and Remark 8.1}, $\HB$ is a squeeze symmetric basis whose fundamental function is of the same order as $(m)_{m=1}^\infty$. By \cite{Woj2003}*{Theorem 10}, $\HB$ is a quasi-greedy basis of $\XX$. Pick a sequence $(Q_j)_{j=1}^\infty$ of pairwise disjoint dyadic cubes such that $k(Q_{j+1})=1+k(Q_j)$, and pick an arbitrary sequence $(\theta_j)_{j=1}^\infty$ in $\Theta_d$. By \cite{DKK2003}*{Corollary 8.6}, $(h_{Q_j,\theta_j})_{j=1}^\infty$ is equivalent to the unit vector system of $\ell_1$. By Lemma~\ref{lem:closel1} and Theorem~\ref{eq:dualizeSS}, the dual basis
\[
\HB^*=(h_{Q,\theta}^*)_{(Q,\theta)\in\Dy\times\Theta_d}.
\]
of $\HB$ satisfies
\[
\leba_m[\HB^*,\XX^*]\approx \sq_m[\HB^*,\XX^*] \approx \sqd_m[\HB^*,\XX^*]\approx \dem_m[\HB^*,\XX^*]
\]
and $\leb_m[\HB^*,\XX^*] =O( \log m)$. We will prove that
\[
\leb_m[\HB^*,\XX^*]\approx \leba_m[\HB^*,\XX^*]\approx \qg_m[\HB^*,\XX^*] \approx \log m.
\]
To that end, it suffices to show that $\log m=O( \usdf[\HB^*,\XX^*](m))$ and $\log m=O( \ucc[\HB^*,\XX^*](m))$.

For each $k\in\NN$ we define $f_k^*\in (\BV(\RR^d))^*$ by
\[
f_k^*(f)= \frac{\partial}{\partial x_1} \left(T_{[0,1]^d,k}(f)\right)\left( \left[ \frac{1}{3},\infty\right)\times\RR^{d-1}\right).
\]
It is clear that $\Vert f_k^*\Vert=\Vert f_k^*|_\XX\Vert$ for all $k\in\NN$, and $\sup_k \Vert f_k^*\Vert <\infty$. Note that for each $j\in\NN\cup\{0\}$ there is a unique dyadic interval $I_j$ with $|I_j|=2^{-j}$ and $1/3\in I_j$. Let $A_k$ (resp.\ $B_k$) be the subset of $\Dy\times\Theta_d$ defined by $(Q,\theta)\in A_k$ (resp. $(Q,\theta)\in B_k$) if and only if $\theta=(1,0,\dots,0)$, $Q=\prod_{i=1}^d J_i\subseteq [0,1)^d$, and $J_1=I_j$ for some even (resp.\ odd) integer $j\in[0,k-1]$. A routine computation yields
\[
f_k^*(h_{Q,\theta})=
\begin{cases}
0 & \text{ if $(Q,\theta)\notin A_k\cup B_k$,}\\
1 & \text{ if $(Q,\theta)\in A_k$, and} \\
-1 & \text{ if $(Q,\theta)\in B_k$} \\
\end{cases}
\]
(cf.\ \cite{GHO2013}*{Example 2}). In other words, $f_k^*|_\XX=\Ind_{A_k}[\HB^*,\XX^*]-\Ind_{B_k}[\HB^*,\XX^*]$. Set $f=\chi_{[0,1/3)\times [0,1)^{d-1}}$. The arguments in \cite{GHO2013}*{Example 2} also give
\[
\Ind_{A_k}[\HB^*,(\BV(\RR^d))^*](f)=\frac{1}{3} \lceil k \rceil, \quad k\in\NN.
\]
Since
\[
|A_k\cup B_k|=\frac{ 2^{(d-1)k}-1}{2^{d-1}-1}
\]
we are done. Note that this yields $\unc_m[\HB,\XX]=\unc_m[\HB^*,\XX^*]\approx \log m$.

\subsection{The Franklin system as a basis of $\VMO$}
As in \S\ref{example:BV}, we denote by $\Dy$ the set consisting of all $d$-dimensional dyadic cubes, $d\in\NN$. The homogeneous Triebel-Lizorkin sequence space $\ring{\tl}_{p,q}^{d}$ of indeces $p,q\in(0,\infty)$ consists of all scalar sequences $f=(a_Q)_{Q\in\Dy}$ for which
\[
\Vert f \Vert_{ \tl_{p,q}^{d}}= \left\Vert \left(\sum_{Q\in\Dy} |Q|^{-q/p} |a_Q|^q \chi_{Q}\right)^{1/q}\right\Vert_p<\infty.
\]
By definition, the unit vector system $\BB=(\ee_Q)_{Q\in\Dy}$ is a normalized unconditional basis of $\ring{\tl}_{p,q}^{d}$. Moreover, it is a democratic (hence greedy) basis whose fundamental function is of the same order as $(m^{1/p})_{m=1}^\infty$ (see \cite{AABW2019}*{\S11.3}). Let $\Dy_0$ denote the set consisting of all dyadic cubes contained in $[0,1]^d$, and consider the subbasis $\BB_0=(\ee_Q)_{Q\in\Dy_0}$ of $\BB$. It is known that certain wavelet bases of homegeneous (resp.\ inhomogenous) Triebel-Lizorkin function spaces $\ring{F}_{p,q}^{s}(\RR^d)$ (resp.\ $F_{p,q}^{s}(\RR^d)$) of smoothness $s\in\RR$ are equivalent to $\BB$ (resp.\ $\BB_0$) regarded as a basis (resp.\ basic sequence) of $\ring{\tl}_{p,q}^{d}$ (see \cite{FrJaWe1991}*{Theorem 7.20} for the homegeneous case and \cite{TriebelIII}*{Theorem 3.5} for the inhomogenous case). In the particular case that $p=1$, $q=2$ and $d=1$, $\BB_0$ is equivalent to both the Franklin system in the Hardy space $H_1$ and the Haar system in the dyadic Hardy space $H_1(\delta)$ (see \cites{Maurey1980,Woj1982}). Consequently, the dual basis of $\BB_0$ is equivalent to both the Franklin system regarded as a basis of $\VMO$ and the Haar system regarded as a basic sequence in dyadic $\BMO$.

Suppose that $1<q<\infty$ and $r=q'$. Consider the space $\ring{\tl}_{\infty,r}^{d}$ consisting of all sequences $f=(a_\lambda)_{Q\in\Dy}$ satisfying the Carleson-type condition
\[
\Vert f \Vert_{ \tl_{\infty,r}^d}= \sup_{P\in\Dy} \left(\frac{1}{|P|} \sum_{\substack{Q\in\Dy\\ Q\subseteq P}} |Q|\, |a_Q|^2 \right)^{1/2}<\infty.
\]
It is known that the dual space of $\ring{\tl}_{1,q}^{d}$ is $\ring{\tl}_{\infty,r}^{d}$ under the natural pairing (see \cite{FrazierJawerth1990}*{Equation (5.2)}. Our analysis of the unit vector system of $\ring{\tl}_{\infty,r}^{d}$ relies on the following lemma.
\begin{lemma}\label{lem:Carleson}
Let $d\in\NN$. There is a constant $C$ such that for every $\AD\subseteq\Dy$ and $P\in\Dy$
\[
L(\AD,P):= \sum_{\substack{Q\in \AD \\ Q\subseteq P}} |Q| \le C |P| \log(1+|\AD|).
\]
Moreover, for every $P\in\Dy$ and $m\in\NN$ there is $\AD\subseteq\Dy$ with $|\AD|= m$ and $ \log(1+m)\le C \, L(\AD,P)$.
\end{lemma}
\begin{proof}
By homogeneity, we can assume that $P=[0,1]^d$. Given $k\in\NN$ we set
\[
\AD_k=\{ Q\in\Dy \colon Q\subseteq[0,1]^d,\; |Q|\ge 2^{-k+1}\}.
\]
We have $L(\AD_k,[0,1]^d)=k$ and
\[
|\AD_k|=m(k):=\frac{2^{dk}-1}{2^d-1}, \quad k\in\NN.
\]
Given $\AD\subseteq\Dy$, let $k\in\NN$ be such that $m(k)\le |\AD| <m(k+1)$. Set $\AD'=\{Q\in\AD \colon |Q|\le 2^{-kd}\}$. We have
\begin{align*}
L(\AD,[0,1]^d)
&\le L(\AD_k,[0,1]^d)+ S(\AD',[0,1]^d)\\
&\le k + 2^{-kd} |\AD'|\\
&\le k +2^{-kd} m(k+1)\\
&\le k+\frac{2^d}{2^d -1}.
\end{align*}
Since $\sup_k (k+(1-2^{-d})^{-1})/\log(1+m(k))<\infty$, we are done. For the `moreover' part, we pick $k\in\NN$ such that
$m(k)\le m < m(k+1)$ and $\AD\supseteq \AD_{k}$ with $|\AD|=m$. Then,
\[
L(\AD,[0,1]^d)\ge L(\AD_{k},[0,1]^d)=k\ge c \log(1+m),
\]
where $c=\inf_k k/\log(m(k+1))>0$.
\end{proof}

Finally we are in a position to estimate the constants of the unit vector system of $\ring{\tl}_{\infty,r}^{d}$. If $\Dy_1\subseteq\Dy$ consists of pairwise disjoint dyadic cubes, then $(\ee_Q)_{Q\in\Dy_1}$ is, when regarded as a basic sequence in $\ring{\tl}_{1,q}^{d}$, isometrically equivalent to the unit vector system of $\ell_1$. Therefore we can apply Lemma~\ref{lem:Carleson} to obtain that $\usdf[\BB,\ring{\tl}_{\infty,r}^{d}]$ and $\usdf[\BB_0,\ring{\tl}_{\infty,r}^{d}]$ grow as $((\log m)^{1/r})_{m=2}^\infty$ Therefore, applying Lemma~\ref{lem:closel1} gives that
\begin{align*}
\leb_m[\BB,\ring{\tl}_{\infty,r}^{d}]
&\approx \leb_m[\BB_0,\ring{\tl}_{\infty,r}^{d}]
\approx\leba_m[\BB,\ring{\tl}_{\infty,r}^{d}]
\approx \leba_m[\BB_0,\ring{\tl}_{\infty,r}^{d}]\\
&\approx ((\log m)^{1/r})_{m=2}^\infty.
\end{align*}



\begin{bibdiv}
\begin{biblist}

\bib{AlbiacAnsorena2017b}{article}{
 author={Albiac, F.},
 author={Ansorena, J.~L.},
 title={Characterization of 1-almost greedy bases},
 date={2017},
 ISSN={1139-1138},
 journal={Rev. Mat. Complut.},
 volume={30},
 number={1},
 pages={13\ndash 24},
 url={https://doi-org.umbral.unirioja.es/10.1007/s13163-016-0204-3},
 review={\MR{3596024}},
}

\bib{AlbiacAnsorena2016}{article}{
 author={Albiac, Fernando},
 author={Ansorena, Jos\'{e}~L.},
 title={Lorentz spaces and embeddings induced by almost greedy bases in
 {B}anach spaces},
 date={2016},
 ISSN={0176-4276},
 journal={Constr. Approx.},
 volume={43},
 number={2},
 pages={197\ndash 215},
 url={https://doi-org/10.1007/s00365-015-9293-3},
 review={\MR{3472645}},
}

\bib{AAB2020}{article}{
 author={Albiac, Fernando},
 author={Ansorena, Jos\'{e}~L.},
 author={Bern\'{a}, Pablo~M.},
 title={Asymptotic greediness of the {H}aar system in the spaces
 {$L_p[0,1]$}, {$1<p<\infty$}},
 date={2020},
 ISSN={0176-4276},
 journal={Constr. Approx.},
 volume={51},
 number={3},
 pages={427\ndash 440},
 url={https://doi.org/10.1007/s00365-019-09466-1},
 review={\MR{4093110}},
}

\bib{AABW2019}{article}{
 author={Albiac, Fernando},
 author={Ansorena, Jos\'{e}~L.},
 author={Bern\'{a}, Pablo~M.},
 author={Wojtaszczyk, Przemys{\l}aw},
 title={Greedy approximation for biorthogonal systems in quasi-banach
 spaces},
 date={2019},
 journal={arXiv e-prints},
 eprint={1903.11651},
 note={Accepted for Publication in Dissertationes Math.},
}

\bib{AACV2019}{article}{
 author={Albiac, Fernando},
 author={Ansorena, Jos\'{e}~L.},
 author={Ciaurri, \'{O}scar},
 author={Varona, Juan~L.},
 title={Unconditional and quasi-greedy bases in {$L_p$} with applications
 to {J}acobi polynomials {F}ourier series},
 date={2019},
 ISSN={0213-2230},
 journal={Rev. Mat. Iberoam.},
 volume={35},
 number={2},
 pages={561\ndash 574},
 url={https://doi.org/10.4171/rmi/1062},
 review={\MR{3945734}},
}

\bib{AADK2019b}{article}{
 author={Albiac, Fernando},
 author={Ansorena, Jos\'{e}~L.},
 author={Dilworth, Stephen~J.},
 author={Kutzarova, Denka},
 title={Building highly conditional almost greedy and quasi-greedy bases
 in {B}anach spaces},
 date={2019},
 ISSN={0022-1236},
 journal={J. Funct. Anal.},
 volume={276},
 number={6},
 pages={1893\ndash 1924},
 url={https://doi-org.umbral.unirioja.es/10.1016/j.jfa.2018.08.015},
 review={\MR{3912795}},
}

\bib{AAW2018b}{article}{
 author={Albiac, Fernando},
 author={Ansorena, Jos\'{e}~L.},
 author={Wallis, Ben},
 title={1-greedy renormings of {G}arling sequence spaces},
 date={2018},
 ISSN={0021-9045},
 journal={J. Approx. Theory},
 volume={230},
 pages={13\ndash 23},
 url={https://doi-org.umbral.unirioja.es/10.1016/j.jat.2018.03.002},
 review={\MR{3800094}},
}

\bib{AAW2020}{article}{
 author={Albiac, Fernando},
 author={Ansorena, Jos\'{e}~Luis},
 author={Wojtaszczyk, Przemys{\l}aw},
 title={On certain subspaces of {$\ell_p$} for {$0<p\leq1$} and their
 applications to conditional quasi-greedy bases in {$p$}-{B}anach spaces},
 date={2021},
 ISSN={0025-5831},
 journal={Math. Ann.},
 volume={379},
 number={1-2},
 pages={465\ndash 502},
 url={https://doi-org.umbral.unirioja.es/10.1007/s00208-020-02069-3},
 review={\MR{4211094}},
}

\bib{AlbiacKalton2016}{book}{
 author={Albiac, Fernando},
 author={Kalton, Nigel~J.},
 title={Topics in {B}anach space theory},
 edition={Second},
 series={Graduate Texts in Mathematics},
 publisher={Springer, [Cham]},
 date={2016},
 volume={233},
 ISBN={978-3-319-31555-3; 978-3-319-31557-7},
 url={https://doi.org/10.1007/978-3-319-31557-7},
 note={With a foreword by Gilles Godefroy},
 review={\MR{3526021}},
}

\bib{AW2006}{article}{
 author={Albiac, Fernando},
 author={Wojtaszczyk, Przemys{\l}aw},
 title={Characterization of 1-greedy bases},
 date={2006},
 ISSN={0021-9045},
 journal={J. Approx. Theory},
 volume={138},
 number={1},
 pages={65\ndash 86},
 url={https://doi.org/10.1016/j.jat.2005.09.017},
 review={\MR{2197603}},
}

\bib{Ansorena2018}{article}{
 author={Ansorena, Jos\'{e}~L.},
 title={A note on subsymmetric renormings of {B}anach spaces},
 date={2018},
 ISSN={1607-3606},
 journal={Quaest. Math.},
 volume={41},
 number={5},
 pages={615\ndash 628},
 url={https://doi-org/10.2989/16073606.2017.1393704},
 review={\MR{3836410}},
}

\bib{Aoki1942}{article}{
 author={Aoki, Tosio},
 title={Locally bounded linear topological spaces},
 date={1942},
 ISSN={0369-9846},
 journal={Proc. Imp. Acad. Tokyo},
 volume={18},
 pages={588\ndash 594},
 url={http://projecteuclid.org/euclid.pja/1195573733},
 review={\MR{14182}},
}

\bib{BBG2017}{article}{
 author={Bern\'{a}, Pablo~M.},
 author={Blasco, \'{O}scar},
 author={Garrig\'{o}s, Gustavo},
 title={Lebesgue inequalities for the greedy algorithm in general bases},
 date={2017},
 ISSN={1139-1138},
 journal={Rev. Mat. Complut.},
 volume={30},
 number={2},
 pages={369\ndash 392},
 url={https://doi.org/10.1007/s13163-017-0221-x},
 review={\MR{3642039}},
}

\bib{BBGHO2018}{article}{
 author={Bern\'{a}, Pablo~M.},
 author={Blasco, Oscar},
 author={Garrig\'{o}s, Gustavo},
 author={Hern\'{a}ndez, Eugenio},
 author={Oikhberg, Timur},
 title={Embeddings and {L}ebesgue-type inequalities for the greedy
 algorithm in {B}anach spaces},
 date={2018},
 ISSN={0176-4276},
 journal={Constr. Approx.},
 volume={48},
 number={3},
 pages={415\ndash 451},
 url={https://doi.org/10.1007/s00365-018-9415-9},
 review={\MR{3869447}},
}

\bib{CRS2007}{article}{
 author={Carro, Mar\'{\i}a~J.},
 author={Raposo, Jos\'{e}~A.},
 author={Soria, Javier},
 title={Recent developments in the theory of {L}orentz spaces and
 weighted inequalities},
 date={2007},
 ISSN={0065-9266},
 journal={Mem. Amer. Math. Soc.},
 volume={187},
 number={877},
 pages={xii+128},
 url={https://doi-org/10.1090/memo/0877},
 review={\MR{2308059}},
}

\bib{CasShu1989}{book}{
 author={Casazza, Peter~G.},
 author={Shura, Thaddeus~J.},
 title={Tsire{l\cprime}son's space},
 series={Lecture Notes in Mathematics},
 publisher={Springer-Verlag, Berlin},
 date={1989},
 volume={1363},
 ISBN={3-540-50678-0},
 url={https://doi-org/10.1007/BFb0085267},
 note={With an appendix by J. Baker, O. Slotterbeck and R. Aron},
 review={\MR{981801}},
}

\bib{Tsirelson1974}{article}{
 author={Cirel\cprime~son, Boris~S.},
 title={It is impossible to imbed {$\ell_{p}$} or {$c_{0}$} into an
 arbitrary {B}anach space},
 date={1974},
 ISSN={0374-1990},
 journal={Funkcional. Anal. i Prilo\v{z}en.},
 volume={8},
 number={2},
 pages={57\ndash 60},
 review={\MR{0350378}},
}

\bib{CDVPX1999}{article}{
 author={Cohen, Albert},
 author={DeVore, Ronald},
 author={Petrushev, Pencho},
 author={Xu, Hong},
 title={Nonlinear approximation and the space {${\rm BV}({\bf R}^2)$}},
 date={1999},
 ISSN={0002-9327},
 journal={Amer. J. Math.},
 volume={121},
 number={3},
 pages={587\ndash 628},
 url={http://muse.jhu.edu/journals/american_journal_of_mathematics/v121/121.3cohen.pdf},
 review={\MR{1738406}},
}

\bib{DKK2003}{article}{
 author={Dilworth, Stephen~J.},
 author={Kalton, Nigel~J.},
 author={Kutzarova, Denka},
 title={On the existence of almost greedy bases in {B}anach spaces},
 date={2003},
 ISSN={0039-3223},
 journal={Studia Math.},
 volume={159},
 number={1},
 pages={67\ndash 101},
 url={https://doi.org/10.4064/sm159-1-4},
 note={Dedicated to Professor Aleksander Pe{\l}czy\'nski on the occasion
 of his 70th birthday},
 review={\MR{2030904}},
}

\bib{DKKT2003}{article}{
 author={Dilworth, Stephen~J.},
 author={Kalton, Nigel~J.},
 author={Kutzarova, Denka},
 author={Temlyakov, Vladimir~N.},
 title={The thresholding greedy algorithm, greedy bases, and duality},
 date={2003},
 ISSN={0176-4276},
 journal={Constr. Approx.},
 volume={19},
 number={4},
 pages={575\ndash 597},
 url={https://doi-org/10.1007/s00365-002-0525-y},
 review={\MR{1998906}},
}

\bib{DKOSZ2014}{article}{
 author={Dilworth, Stephen~J.},
 author={Kutzarova, Denka},
 author={Odell, Edward~W.},
 author={Schlumprecht, Thomas},
 author={Zs\'{a}k, Andr\'{a}s},
 title={Renorming spaces with greedy bases},
 date={2014},
 ISSN={0021-9045},
 journal={J. Approx. Theory},
 volume={188},
 pages={39\ndash 56},
 url={https://doi.org/10.1016/j.jat.2014.09.001},
 review={\MR{3274228}},
}

\bib{DKW2002}{article}{
 author={Dilworth, Stephen~J.},
 author={Kutzarova, Denka},
 author={Wojtaszczyk, Przemys{\l}aw},
 title={On approximate {$l_1$} systems in {B}anach spaces},
 date={2002},
 ISSN={0021-9045},
 journal={J. Approx. Theory},
 volume={114},
 number={2},
 pages={214\ndash 241},
 url={https://doi.org/10.1006/jath.2001.3641},
 review={\MR{1883407}},
}

\bib{DOSZ2011}{article}{
 author={Dilworth, Stephen~J.},
 author={Odell, Edward~W.},
 author={Schlumprecht, Thomas},
 author={Zs\'{a}k, Andr\'{a}s},
 title={Renormings and symmetry properties of 1-greedy bases},
 date={2011},
 ISSN={0021-9045},
 journal={J. Approx. Theory},
 volume={163},
 number={9},
 pages={1049\ndash 1075},
 url={https://doi.org/10.1016/j.jat.2011.02.013},
 review={\MR{2832742}},
}

\bib{DSBT2012}{article}{
 author={Dilworth, Stephen~J.},
 author={Soto-Bajo, Mois\'es},
 author={Temlyakov, Vladimir~N.},
 title={Quasi-greedy bases and {L}ebesgue-type inequalities},
 date={2012},
 ISSN={0039-3223},
 journal={Studia Math.},
 volume={211},
 number={1},
 pages={41\ndash 69},
 url={https://doi-org/10.4064/sm211-1-3},
 review={\MR{2990558}},
}

\bib{Donoho1993}{article}{
 author={Donoho, David~L.},
 title={Unconditional bases are optimal bases for data compression and
 for statistical estimation},
 date={1993},
 ISSN={1063-5203},
 journal={Appl. Comput. Harmon. Anal.},
 volume={1},
 number={1},
 pages={100\ndash 115},
 url={https://doi.org/10.1006/acha.1993.1008},
 review={\MR{1256530}},
}

\bib{DRS1969}{article}{
 author={Duren, Peter~L.},
 author={Romberg, Bernhard~W.},
 author={Shields, Allen~L.},
 title={Linear functionals on {$H^{p}$} spaces with {$0<p<1$}},
 date={1969},
 ISSN={0075-4102},
 journal={J. Reine Angew. Math.},
 volume={238},
 pages={32\ndash 60},
 review={\MR{259579}},
}

\bib{FigielJohnson1974}{article}{
 author={Figiel, Tadeusz},
 author={Johnson, William~B.},
 title={A uniformly convex {B}anach space which contains no {$l_{p}$}},
 date={1974},
 ISSN={0010-437X},
 journal={Compositio Math.},
 volume={29},
 pages={179\ndash 190},
 review={\MR{355537}},
}

\bib{FrazierJawerth1990}{article}{
 author={Frazier, Michael},
 author={Jawerth, Bj\"{o}rn},
 title={A discrete transform and decompositions of distribution spaces},
 date={1990},
 ISSN={0022-1236},
 journal={J. Funct. Anal.},
 volume={93},
 number={1},
 pages={34\ndash 170},
 url={https://doi.org/10.1016/0022-1236(90)90137-A},
 review={\MR{1070037}},
}

\bib{FrJaWe1991}{book}{
 author={Frazier, Michael},
 author={Jawerth, Bj\"{o}rn},
 author={Weiss, Guido},
 title={Littlewood-{P}aley theory and the study of function spaces},
 series={CBMS Regional Conference Series in Mathematics},
 publisher={Published for the Conference Board of the Mathematical Sciences,
 Washington, DC; by the American Mathematical Society, Providence, RI},
 date={1991},
 volume={79},
 ISBN={0-8218-0731-5},
 url={https://doi.org/10.1090/cbms/079},
 review={\MR{1107300}},
}

\bib{GHO2013}{article}{
 author={Garrig\'os, Gustavo},
 author={Hern\'{a}ndez, Eugenio},
 author={Oikhberg, Timur},
 title={Lebesgue-type inequalities for quasi-greedy bases},
 date={2013},
 ISSN={0176-4276},
 journal={Constr. Approx.},
 volume={38},
 number={3},
 pages={447\ndash 470},
 url={https://doi-org/10.1007/s00365-013-9209-z},
 review={\MR{3122278}},
}

\bib{Gogyan2016}{article}{
 author={Gogyan, Smbat~L.},
 title={The quasi-greedy property of subsystems of the multivariate
 {H}aar system},
 date={2016},
 ISSN={1607-0046},
 journal={Izv. Ross. Akad. Nauk Ser. Mat.},
 volume={80},
 number={3},
 pages={23\ndash 30},
 url={https://doi.org/10.4213/im8346},
 review={\MR{3507385}},
}

\bib{Grafakos2004}{book}{
 author={Grafakos, Loukas},
 title={Classical and modern {F}ourier analysis},
 publisher={Pearson Education, Inc., Upper Saddle River, NJ},
 date={2004},
 ISBN={0-13-035399-X},
 review={\MR{2449250}},
}

\bib{GPV1990}{article}{
 author={Guadalupe, Jos\'{e}~J.},
 author={P\'{e}rez, Mario},
 author={Varona, Juan~L.},
 title={Weak behaviour of {F}ourier-{J}acobi series},
 date={1990},
 ISSN={0021-9045},
 journal={J. Approx. Theory},
 volume={61},
 number={2},
 pages={222\ndash 238},
 url={https://doi.org/10.1016/0021-9045(90)90005-B},
 review={\MR{1050619}},
}

\bib{KadecPel1962}{article}{
 author={Kadets, Mikhail~I.},
 author={Pe{\l}czy{\'n}ski, Aleksander},
 title={Bases, lacunary sequences and complemented subspaces in the
 spaces {$L_{p}$}},
 date={1961/1962},
 ISSN={0039-3223},
 journal={Studia Math.},
 volume={21},
 pages={161\ndash 176},
 review={\MR{0152879}},
}

\bib{KoTe1999}{article}{
 author={Konyagin, Sergei~V.},
 author={Temlyakov, V.N.},
 title={A remark on greedy approximation in {B}anach spaces},
 date={1999},
 ISSN={1310-6236},
 journal={East J. Approx.},
 volume={5},
 number={3},
 pages={365\ndash 379},
 review={\MR{1716087}},
}

\bib{Lin1964}{article}{
 author={Lindenstrauss, Joram},
 title={On a certain subspace of {$l_{1}$}},
 date={1964},
 ISSN={0001-4117},
 journal={Bull. Acad. Polon. Sci. S\'{e}r. Sci. Math. Astronom. Phys.},
 volume={12},
 pages={539\ndash 542},
 review={\MR{174963}},
}

\bib{LinTza1977}{book}{
 author={Lindenstrauss, Joram},
 author={Tzafriri, Lior},
 title={Classical {B}anach spaces. {I}},
 publisher={Springer-Verlag, Berlin-New York},
 date={1977},
 ISBN={3-540-08072-4},
 note={Sequence spaces, Ergebnisse der Mathematik und ihrer
 Grenzgebiete, Vol. 92},
 review={\MR{0500056}},
}

\bib{Maurey1980}{article}{
 author={Maurey, Bernard},
 title={Isomorphismes entre espaces {$H_{1}$}},
 date={1980},
 ISSN={0001-5962},
 journal={Acta Math.},
 volume={145},
 number={1-2},
 pages={79\ndash 120},
 url={https://doi.org/10.1007/BF02414186},
 review={\MR{586594}},
}

\bib{MPS1981}{article}{
 author={McGehee, O.~Carruth},
 author={Pigno, Louis},
 author={Smith, Brent},
 title={Hardy's inequality and the {$L^{1}$} norm of exponential sums},
 date={1981},
 ISSN={0003-486X},
 journal={Ann. of Math. (2)},
 volume={113},
 number={3},
 pages={613\ndash 618},
 url={https://doi.org/10.2307/2007000},
 review={\MR{621019}},
}

\bib{Mehta2015}{article}{
 author={Mehta, Harsh},
 title={The {$L^1$} norms of de la {V}all\'{e}e {P}oussin kernels},
 date={2015},
 ISSN={0022-247X},
 journal={J. Math. Anal. Appl.},
 volume={422},
 number={2},
 pages={825\ndash 837},
 url={https://doi.org/10.1016/j.jmaa.2014.09.018},
 review={\MR{3269485}},
}

\bib{Oswald2001}{article}{
 author={Oswald, Peter},
 title={Greedy algorithms and best {$m$}-term approximation with respect
 to biorthogonal systems},
 date={2001},
 ISSN={1069-5869},
 journal={J. Fourier Anal. Appl.},
 volume={7},
 number={4},
 pages={325\ndash 341},
 url={https://doi.org/10.1007/BF02514500},
 review={\MR{1836816}},
}

\bib{Pollard1949}{article}{
 author={Pollard, Harry},
 title={The mean convergence of orthogonal series. {III}},
 date={1949},
 ISSN={0012-7094},
 journal={Duke Math. J.},
 volume={16},
 pages={189\ndash 191},
 url={http://projecteuclid.org/euclid.dmj/1077475368},
 review={\MR{28459}},
}

\bib{Rolewicz1957}{article}{
 author={Rolewicz, Stefan},
 title={On a certain class of linear metric spaces},
 date={1957},
 journal={Bull. Acad. Polon. Sci. Cl. III.},
 volume={5},
 pages={471\ndash 473, XL},
 review={\MR{0088682}},
}

\bib{Rudin1959}{article}{
 author={Rudin, Walter},
 title={Some theorems on {F}ourier coefficients},
 date={1959},
 ISSN={0002-9939},
 journal={Proc. Amer. Math. Soc.},
 volume={10},
 pages={855\ndash 859},
 url={https://doi.org/10.2307/2033608},
 review={\MR{116184}},
}

\bib{Tembook}{book}{
 author={Temlyakov, Vladimir},
 title={Greedy approximation},
 series={Cambridge Monographs on Applied and Computational Mathematics},
 publisher={Cambridge University Press, Cambridge},
 date={2011},
 volume={20},
 ISBN={978-1-107-00337-8},
 url={https://doi.org/10.1017/CBO9780511762291},
 review={\MR{2848161}},
}

\bib{Temlyakov2011}{misc}{
 author={Temlyakov, Vladimir~M.},
 date={2011},
 note={{Question posed during the \emph{Concentration week on greedy
 algorithms in Banach spaces and compressed sensing} held on 18--22 July,
 2011, at Texas A\&M University.}},
}

\bib{Temlyakov1998}{article}{
 author={Temlyakov, Vladimir~N.},
 title={The best {$m$}-term approximation and greedy algorithms},
 date={1998},
 ISSN={1019-7168},
 journal={Adv. Comput. Math.},
 volume={8},
 number={3},
 pages={249\ndash 265},
 url={https://doi.org/10.1023/A:1018900431309},
 review={\MR{1628182}},
}

\bib{Temlyakov1998b}{article}{
 author={Temlyakov, Vladimir~N.},
 title={Greedy algorithm and {$m$}-term trigonometric approximation},
 date={1998},
 ISSN={0176-4276},
 journal={Constr. Approx.},
 volume={14},
 number={4},
 pages={569\ndash 587},
 url={https://doi.org/10.1007/s003659900090},
 review={\MR{1646563}},
}

\bib{Temlyakov2015}{article}{
 author={Temlyakov, Vladimir~N.},
 title={Constructive sparse trigonometric approximations and other
 problems for functions with mixed smoothness},
 date={2015},
 ISSN={0368-8666},
 journal={Mat. Sb.},
 volume={206},
 number={11},
 pages={131\ndash 160},
 url={https://doi.org/10.4213/sm8466},
 review={\MR{3438571}},
}

\bib{TemYangYe2}{article}{
 author={Temlyakov, Vladimir~N.},
 author={Yang, Mingrui},
 author={Ye, Peixin},
 title={Greedy approximation with regard to non-greedy bases},
 date={2011},
 ISSN={1019-7168},
 journal={Adv. Comput. Math.},
 volume={34},
 number={3},
 pages={319\ndash 337},
 url={https://doi.org/10.1007/s10444-010-9155-2},
 review={\MR{2776447}},
}

\bib{TemYangYe1}{article}{
 author={Temlyakov, Vladimir~N.},
 author={Yang, Mingrui},
 author={Ye, Peixin},
 title={Lebesgue-type inequalities for greedy approximation with respect
 to quasi-greedy bases},
 date={2011},
 ISSN={1310-6236},
 journal={East J. Approx.},
 volume={17},
 number={2},
 pages={203\ndash 214},
 review={\MR{2883511}},
}

\bib{TriebelIII}{book}{
 author={Triebel, Hans},
 title={Theory of function spaces. {III}},
 series={Monographs in Mathematics},
 publisher={Birkh\"{a}user Verlag, Basel},
 date={2006},
 volume={100},
 ISBN={978-3-7643-7581-2; 3-7643-7581-7},
 review={\MR{2250142}},
}

\bib{Veprintsev2016}{article}{
 author={Veprintsev, Roman~A.},
 title={On Paley-type and Hausdorff-Young-Paley-type inequalities for
 Jacobi expansions},
 date={2016},
 journal={arXiv e-prints},
 eprint={1603.03965},
}

\bib{Woj1982}{article}{
 author={Wojtaszczyk, Przemys{\l}aw},
 title={The {F}ranklin system is an unconditional basis in {$H_{1}$}},
 date={1982},
 ISSN={0004-2080},
 journal={Ark. Mat.},
 volume={20},
 number={2},
 pages={293\ndash 300},
 url={https://doi.org/10.1007/BF02390514},
 review={\MR{686177}},
}

\bib{Woj2000}{article}{
 author={Wojtaszczyk, Przemys{\l}aw},
 title={Greedy algorithm for general biorthogonal systems},
 date={2000},
 ISSN={0021-9045},
 journal={J. Approx. Theory},
 volume={107},
 number={2},
 pages={293\ndash 314},
 url={https://doi-org/10.1006/jath.2000.3512},
 review={\MR{1806955}},
}

\bib{Woj2003}{article}{
 author={Wojtaszczyk, Przemys{\l}aw},
 title={Projections and non-linear approximation in the space {$\BV(\RR^d)$}},
 date={2003},
 ISSN={0024-6115},
 journal={Proc. London Math. Soc. (3)},
 volume={87},
 number={2},
 pages={471\ndash 497},
 url={https://doi.org/10.1112/S0024611503014084},
 review={\MR{1990936}},
}

\bib{Woj2014}{article}{
 author={Wojtaszczyk, Przemys{\l}aw},
 title={On left democracy function},
 date={2014},
 ISSN={0208-6573},
 journal={Funct. Approx. Comment. Math.},
 volume={50},
 number={2},
 pages={207\ndash 214},
 url={https://doi-org/10.7169/facm/2014.50.2.1},
 review={\MR{3229057}},
}

\end{biblist}
\end{bibdiv}



\end{document}